\documentclass[a4paper,oneside,english,reqno]{amsart}
\usepackage[latin9]{inputenc}
\usepackage[margin=3cm]{geometry}

\usepackage{amsmath}
\usepackage[usenames]{color}
\usepackage{hyperref}
\usepackage{bbm, dsfont}

\pdfpageheight\paperheight
\pdfpagewidth\paperwidth

\newcommand{\dd}{{\rm d}}
\newcommand{\e}[1]{\,{\rm e}^{#1}\,}
\newcommand{\ii}{{\rm i}}

\newcommand{\caN}{{\mathcal N}}
\newcommand{\caO}{{\mathcal O}}

\newcommand{\bbE}{{\mathbb E}}

\newcommand{\bbN}{{\mathbb N}}

\newcommand{\bbP}{{\mathbb P}}

\newcommand{\bbR}{{\mathbb R}}

\newcommand{\bsq}{{\boldsymbol q}}

\newcommand{\bss}{{\boldsymbol s}}

\numberwithin{equation}{section}
\numberwithin{figure}{section}
\theoremstyle{plain}
\newtheorem{thm}{\protect\theoremname}[section]
  \theoremstyle{plain}
  \newtheorem{lem}[thm]{\protect\lemmaname}
  \theoremstyle{plain}
  \newtheorem{prop}[thm]{\protect\propositionname}
   
  \theoremstyle{remark}
   \newtheorem{rem}[thm]{Remark}

\makeatother

\usepackage{babel}
  \providecommand{\lemmaname}{Lemma}
  \providecommand{\propositionname}{Proposition}
\providecommand{\theoremname}{Theorem}

\newcommand{\N}{\mathbb{N}}

\newcommand{\E}[1]{\mathbb{E}\left[#1\right]}

\begin{document}

\title[Random permutations without macroscopic cycles]{Random permutations without macroscopic cycles}

\author{Volker Betz \and Helge Sch\"afer \and Dirk Zeindler}


\begin{abstract}
We consider uniform random permutations of length $n$
conditioned to have 
 no cycle longer than $n^\beta$ with $0<\beta<1$,
in the limit of large $n$. 
Since in unconstrained  uniform random permutations most of the indices 
are in cycles of macroscopic length, this is a singular conditioning in 
the limit. Nevertheless, we obtain a fairly complete picture about the 
cycle number distribution at various lengths. 
Depending on the scale at which cycle numbers are studied,  our results 
include Poisson convergence, 
a central limit theorem, a shape theorem
and two different functional central limit theorems. 
\end{abstract}

\maketitle

\section{Introduction}

Uniform random permutations are among the oldest and best understood models
of probability theory. One of their most prominent properties is that 
almost all indices are in macroscopic cycles: for all $\varepsilon > 0$, 
the probability that a given 
index of a uniform permutation of length $n$ is in a cycle of length less 
than $n\varepsilon$ converges to $\varepsilon$ as 
$n \to \infty$. 
Classical results about uniform random permutations include 
the convergence of the renormalized cycle structure towards a 
Poisson-Dirichlet distribution {\cite{Ki77, ShVe77}}, convergence of joint 
cycle numbers towards independent Poisson random variables in total 
variation distance \cite{AT92}, and a central limit theorem for 
cumulative cycle numbers \cite{DP85}.

Going beyond uniform random permutations, natural models are those where the probability
measure is still invariant under conjugation with a transposition, i.e. it depends only on the cycle structure. One variety of such
models are those with cycle weights, including the Ewens model \cite{Ewens} with applications in genetics, 
or more general cycle weight models \cite{BUV,BZ15,ElPel,ErUe11}
with applications in quantum many body systems \cite{BU1,BU2}.
Another variant is to condition on the absence of cycles of a given length. When the set $A\subset \mathbb{N}$ of forbidden cycle lengths is independent
of the permutation length $n$, this goes under the name of $A$-permutations \cite{YAK07,YAK08}. The case where the forbidden set of cycle lengths depends on $n$ is less well understood.

Our results can be paraphrased as follows: Let $a_{1},a_{2}\in\left(0,1\right)$ and fix a sequence $\alpha(n)$ with
\begin{equation}
\label{condition on alpha}
n^{a_1} \leq \alpha(n) \leq n^{a_2}.
\end{equation}
We consider the uniform measure 
on permutations of length $n$ with cycles of length less than $\alpha(n)$.
We will see in Theorem~\ref{main thm 2}, as $n\to\infty$, cycles of order  
$o(\alpha(n) / \log n)$ have the same asymptotic behaviour as on the full symmetric group.
 At the scale $\alpha(n) / \log n$, the influence of the restriction starts to manifest:
 If $C_m$ denotes the number of cycles of length $m$, then, for $m_n=\text{const}\cdot\alpha(n) / \log n$, $\E{C_{m_n}}$ converges to zero at a slower rate than for uniform permutations, as $n\to\infty$, see Section~\ref{expcyclen}.
 At the scale $c \alpha(n)$, $0 \leq c < 1$, the influence of the restriction becomes even stronger.
 If we have $\alpha(n) = o(\sqrt{n})$ then $\E{C_{m(n)}}\to \infty$ for $m(n)=c \alpha(n)$ for $c$  sufficiently close to $1$. In this case, a central limit theorem holds for $C_{m(n)}$, see Theorem~\ref{main thm 3}. This behaviour is new and cannot be observed for classical random permutations.
 Finally, we consider the scale $\alpha(n)$.
 We show that almost all cycles live at this scale and the limit as $n\to\infty$ of
 the fraction of cycles larger than $\alpha(n)\left(1 - \frac{\epsilon}{\log n}\right) $ tends to $1$ as $\epsilon\to\infty$.  
%
Also we show that at this scale, the cumulative cycle 
numbers satisfy a limit shape theorem, and their fluctuations 
around that limit shape satisfy a functional central limit 
theorem to the Brownian bridge, see Theorems~\ref{thm:Shape} and~\ref{thm:Fluc}.

The proofs of our results are based on the saddle point method
of asymptotic analysis. In particular, we benefit from the 
precise estimates given by Manstavicius and Petuchovas 
\cite{manstavivcius2016local} for the probability that 
an unconstrained permutation has no long cycles. While it is
clear that such results must be useful for our purposes, 
it is surprising that they, and extensions of the methods by 
which they are proved, provide such a complete picture of the 
situation. 

Let us give an outline of the paper. 
In Section \ref{results}, we state our assumptions and 
results. Section \ref{sec:Generating} discusses the relevant
saddle point method in our context and presents a general 
asymptotic equality which is at the base of almost all proofs of our 
main results. Section \ref{sec:proofs} then contains those 
proofs.

\section{Main results} \label{results}

\subsection{Notation and standing assumptions}
%
For $n\in\mathbb{N}$, let $S_{n,\alpha}$ 
be the set of permutations of length $n$ where all cycles have length 
$\alpha(n)$ or less, and let 
$\mathbb{P}_{n,\alpha}$ be the uniform measure on 
$S_{n,\alpha}$. We write $\mathbb{E}_{n,\alpha}$ for the 
expectation with respect to $\mathbb{P}_{n,\alpha}$. 
Furthermore, we denote by  $\mathbb{P}_{n}$ the uniform measure on $S_n$ and by 
$\mathbb{E}_{n}$ the expectation on $S_n$ with respect to $\mathbb{P}_{n}$. 
 We study here the (joint) distribution of the 
random variables 
$C_m$, where $C_m=C_m(\sigma)$ denotes the number of cycles of length $m$ in the cycle decomposition of a permutation $\sigma$.
The index $m$ will often depend on $n$ and $\alpha(n)$, but we 
sometimes omit this dependence when it is clear from the context.
%
%
When two sequences $(a_n)$ and $(b_n)$ are asymptotically 
equivalent, i.e.\ if $\lim_{n\to\infty} a_n/b_n = 1$, we write 
$a_n \sim b_n$. We also use the usual $\caO$ and $o$ notation,
i.e. $f(n) = \caO(g(n))$ means that there exists some constant 
$c > 0$ so that $|f(n)| \leq c |{g(n)}|$ for large $n$,
while $f(n) = o(g(n))$ means that for all $c>0$ there exists 
$n_c \in \bbN$ so that the inequality holds for all 
$n > n_c$.

\subsection{Cycle counts} 
\label{expcyclen}
The most basic characteristics of the $C_m$ 
are their expected values. Let $x_{n,\alpha}$ be the unique 
positive 
solution of the equation 
\begin{equation}
n=\sum_{j=1}^{\alpha}x_{n,\alpha}^{j},
\label{eq:StaSad}
\end{equation}
and define 
\begin{align}
\mu_m\left(n\right):=\frac{x_{n,\alpha}^{m}}{m}.
\label{eq:def_mu_n}
\end{align}
\begin{prop} 
\label{mean}
For all sequences $m = (m(n))_{n \in \bbN}$ 
with $m(n) \leq \alpha(n)$ for all $n$, we have 
\[
\bbE_{n,\alpha(n)} [C_{m(n)}] \sim 
\mu_{m(n)}(n)
\]
as $n \to \infty$. Furthermore, 
\begin{equation}
\label{mu asymptotics}	
	\frac{1}{m}\log (m \mu_{m}) =
	\log x_{n,\alpha} =
	 \frac{1}{\alpha} \Big(  
	\log \tfrac{n}{\alpha} + \log\log \tfrac{n}{\alpha} 
	+ \caO \big( 
	\tfrac{\log\log n}{\log n} \big) \Big) 
\end{equation}
for large $n$.
\end{prop}
%
%
An example illustrates the amount of 
information that we can already extract from Proposition~\ref{mean}.
Recall that for uniform permutations, $\bbE_n [C_m]=\frac{1}{m}$ for all $m\leq n$ \cite[Lemma 1.1]{ABT02}.  
We fix 
$\beta \in (0,1)$ and let $\alpha(n) = n^\beta$.
Equation (\ref{mu asymptotics}) then reads 
\[
\log ( m \mu_m) = m n^{-\beta} \Big( (1 - \beta) \log n + 
\log \log n + \log (1-\beta) + o(1) \Big).
\]
We now have the following asymptotic regimes:\\ 
(1)
	For $m(n) = o(n^\beta/\log n)$, we have 
	$\lim_{n\to\infty} \mu_{m(n)}(n) m(n) = 1$. Thus we have
	\begin{align}
	 \bbE_{n,\alpha(n)} [C_{m(n)}] \sim  \frac{1}{m(n)}= \bbE_n[C_{m(n)}] .
	\end{align}
        In particular, the limiting behavior  is independent of $\beta$. 
	We call this the classical regime.\\
(2)
	For $m(n) = y\,n^\beta/\log n$ with $y>0$, we get $\lim_{n\to\infty} \mu_{m(n)}(n) m(n) = \e{y(1-\beta)}$. Thus
	\begin{align}
	 \bbE_{n,\alpha(n)} [C_{m(n)}] \sim \frac{\e{y(1-\beta)}}{m(n)} = \e{y(1-\beta)} \bbE_n[C_{m(n)}].
	\end{align}
	So in this regime, the number of cycles converges to zero more slowly than in unconstrained permutations.
	We therefore see that the constraint becomes visible in this region.
	Explicitly, we get 
	$$\mu_{m(n)}(n)  \sim  \frac{\log n}{y n^\beta} \e{y(1-\beta)}.$$
	The right-hand side above is minimal for $y = 1/(1-\beta)$ and then has the value $\mu_{m(n)}(n) \sim {\rm e} (1-\beta) \frac{\log n}{n^\beta}$.
	Also, we have that $\mu_{m(n)}(n)$ is increasing as a function of $y$ for $y \geq 1/(1-\beta)$.\\
(3)
	The next regime occurs when we put $m = c n^\beta$ for $0 < c \leq 1$. 
	Then 
	\[
	\log \mu_m = (c(1-\beta) - \beta) \log n + 
	c \log \log n + c \log (1-\beta) - \log c + o(1).
	\]
	We see that $\mu_m \to 0$ when $c < \beta/(1-\beta)$, and 
	$\mu_m \to \infty$ when $c \geq \beta/(1-\beta)$. 
	So on this scale, the transition from finite 
	cycle counts to infinite ones occurs. 
	However, the case of infinite 
	cycle counts can only occur if there exists $c \in (0,1]$
	with 
	$c \geq \beta/(1-\beta)$, which means that 
	$\beta \leq 1/2$. 
	This can be explained intuitively as follows:
	Since the maximal cycle length is $n^\beta$, a permutation $\sigma\in S_{n,\alpha}$ has (at least) $n/n^\beta =n^{1-\beta}$ cycles.
	If $\beta>1/2$ then $n^\beta \gg n^{1-\beta}$ and thus there are more cycle lengths available than cycles.
	So there is no need for too many cycles to have the same length. 
	The situation is reversed when $\beta < 1/2$.
	We have in this case $n^\beta \ll n^{1-\beta}$, and thus there are always more cycles than available cycle lengths.
	The pigeon-hole principle now implies that at least $n^{1-\beta}/n^{\beta}=n^{1-2\beta}$ cycles have the same length.
	Since $\beta < 1/2$, we have $n^{1-2\beta}\to\infty$ and thus there has to be $m=m(n)$ such that $C_m\to\infty$. \\[1mm]
We will now investigate the joint distributions 
of the random variables $C_j$. We start
with the strongest result, which also has the most 
restrictive assumptions. Recall that the total variation distance of two probability 
measures $\bbP$ and 
$\tilde \bbP $ on a discrete probability space $\Omega$ is simply given by 
$\| \bbP - \tilde \bbP \|_{\rm TV} = \sum_{\omega \in \Omega} (\bbP(\omega) - \tilde \bbP(\omega))_+$. 
\begin{thm} \label{main thm 2}
	Let $b = (b(n))_n$ be a sequence so that 
	$b(n) = o \big( \alpha(n) (\log n)^{-1}\big)$. Let 
	$\bbP_{n,b(n),\alpha}$ be the distribution of $(C_1, \ldots C_{b(n)})$
	under $\bbP_{n,\alpha}$, and let $\tilde \bbP_{b(n)}$ 
	be the distribution
	of independent Poisson-distributed random variables 
	$(Z_{1}, \ldots Z_{b(n)})$ with 
	$\tilde \bbE _{b(n)}(Z_{j}) = \frac{1}{j}$ 
	for all $j \leq b(n)$. Then there exists $c<\infty$ so that for all 
	$n \in \bbN$, we have  
	\[
	\| \bbP_{n,b(n),\alpha} - \tilde \bbP_{b(n)} \|_{\rm TV} \leq c 
	\left( \frac{\alpha(n)}{n} + b(n) \frac{\log n}{\alpha(n)}
	\right).
	\]
\end{thm}
%
Let $\bbP_{n,b(n)}$ be the distribution of $(C_1, \ldots C_{b(n)})$ under $\bbP_n$.
Then it was proven by Arratia and Tavar{\'e} in  \cite[Theorem~2]{AT92b} that $\| \bbP_{n,b(n)} - \tilde \bbP_{b(n)} \|_{\rm TV} \to 0$  iff $b(n) = o(n)$.
Thus the cycles of lengths $ o \big( \alpha(n) (\log n)^{-1}\big)$ have a similar behaviour under $\bbP_n$ and under $\bbP_{n,\alpha}$.
Furthermore, Arratia and Tavar{\'e} show in \cite[Theorem~2]{AT92} that 
there exists a function $F$ with $\log F(x)\sim -x\log x \text{ as } x\to\infty$ so that $\| \bbP_{n,b(n)} - \tilde \bbP_{b(n)} \|_{\rm TV} \leq F(n/b(n))$. 
This fast decay rate appears to be  special for the uniform measure. 
The decay rate for all other known measures is at most algebraically fast, including the case we study in this paper.

We can slightly relax the condition 
$b(n) = o (\alpha(n) (\log n)^{-1})$ in Theorem~\ref{main thm 2} if we only consider 
convergence of finite-dimensional distributions. What is more, 
we can in this case apply a 'tilt' as we would do in large 
deviations theory in order to get a better understanding of those cases where $\mu_m \to 0$ in Proposition \ref{mean}. 
For $\nu \in \bbR_0^+$, 
consider the {\em tilted cycle numbers} 
$C_k^{(\nu)}$ with distribution 
\begin{align*}
\mathbb{P}\left[C_{m_1}^{(\nu_{1})}=l_1, \ldots, C_{m_k}^{(\nu_{k})}=l_k\right]=\frac{1}{Z}\left(\prod_{j=1}^{k}\frac{\mathrm{e}^{\nu_j}}{\nu_j^{l_j}}\right)\mathbb{P}_{n,\alpha}\left[C_{m_1}=l_1, \ldots, C_{m_k}=l_k\right]
\end{align*}
for all $l_1,...,l_k\in\N_0$, where $Z$ is a normalizing constant.
\begin{thm} 
\label{main thm 1}
	Let $(m_1(n))_n, \ldots, (m_k(n))_n$ be sequences
	with $m_k(n) \leq \alpha(n)$ for all $n$ 
	and $m_i(n)\neq m_j(n)$ for $i\neq j$.
	  Assume that for all $j \leq k$, 
	\begin{equation} \label{limsup condition}
	\limsup_{n\to\infty} \mu_{m_j(n)}(n) < \infty.
	\end{equation}
	Then, as $n\to\infty$, 
	\[
	\Big(C_{m_1}^{(\mu_{m_1})}, \ldots, 
	C_{m_k}^{(\mu_{m_k})}\Big) \stackrel{d}{\longrightarrow} 
	(Z_1, \ldots, Z_k),
	\]
	where the $Z_j$ are independent Poisson distributed 
	random variables with 
	 parameter $1$.
%
\end{thm}
From equation \eqref{mu asymptotics} and our assumptions on $\alpha(n)$ in \eqref{condition on alpha}, 
it follows that a sufficient condition for \eqref{limsup condition} is that 
$m_j(n) \leq c \alpha(n)$ for some $c < \frac{a_1}{1 - a_1}$ with $a_1$ as in \eqref{condition on alpha}.
The case when $m_j(n)$ converges to a limit is already covered by Theorem~\ref{main thm 2}.
The most interesting applications of Theorem~\ref{main thm 1} are in the situation when $\mu_{m_j}$ converges to a limit
while $m_j \to \infty$ as $n\to\infty$.
For instance, if $\mu_m \to 0$, $C_m$ converges in distribution
to the trivial Poisson distribution with parameter $0$, but just like it is the case in large deviations 
theory, the tilt allows us to extract much more information about this convergence. 
We have in particular that for all $j\in\N_0$, the probability 
$\bbP_{n,\alpha}[C_{m} = j]$ decays like $\mu_{m}^{-j}$.

We now treat the case of diverging expected cycle numbers. 
Here, the standard rescaling leads to a central limit theorem.
\begin{thm}
\label{main thm 3}
Let $(m_1(n))_n, \ldots, (m_k(n))_n$ be sequences
	with $m_j(n) \leq \alpha(n)$ for all $n$ and all $j$
	and $m_i(n)\neq m_j(n)$ for $i\neq j$.
	Assume that 
	$\mu_{m_j(n)}(n)\rightarrow\infty$ as $n \to \infty$ for 
	all $j$. Assume finally that in 
	\eqref{condition on alpha}, we have $a_1 > 1/7$. 
	Define
\begin{align*}
\widetilde{C}_{m_{j}}
:=
\frac{C_{m_{j}}-\mu_{m_j}}{\sqrt{\mu_{m_j}}}.
\end{align*}
%
Then
\begin{align*}
 \bigl(\widetilde{C}_{m_{1}},\ldots, \widetilde{C}_{m_{k}} \bigr)  
 \stackrel{d}{\longrightarrow} 
 (N_1,\ldots, N_k) \qquad \text{as } n \to \infty,
\end{align*}
where $(N_j)_{j=1}^k$ are independent, standard normal 
distributed random variables.
\end{thm}
The condition 
$\alpha\left(n\right) \geq n^{\frac{1}{7}+\delta}$ is a technical one, 
and making it allows to avoid significant technical complications. A forthcoming paper will show that the theorem holds under 
condition \eqref{condition on alpha}.
%
%

\subsection{Cumulative cycle numbers}
Let 
\[
K_m = \sum_{j=1}^m C_j,
\]
be the number of cycles with lengths less than $m$. 
Since no cycle can be 
larger than $\alpha(n)$, the total number of cycles  
$K_{\alpha(n)}$ is at least $\geq n/\alpha(n)$.
In \cite{BS17} it is shown that indeed  
$K_{\alpha(n)} \sim \frac{n}{\alpha(n)}$, and so the random 
variable $\frac{K_{m(n)}}{n / \alpha(n)}$ gives the fraction of 
cycles that have length up to $m(n)$. 
The regime in which this fraction converges to a finite limit
will be given by 
\begin{equation} 
b_t(n)
:=
\max\left\{ \alpha\left(n\right)+\left\lfloor \log\left(t\right)\frac{\alpha\left(n\right)}{\log\left(\frac{n}{\alpha\left(n\right)}\right)}\right\rfloor ,0\right\}, \qquad 0\leq t \leq 1.
\label{eq:kDef}	
\end{equation}
We  have the following limit shape 
of the random function $t \mapsto K_{b_t(n)}$: 
\begin{thm}
\label{thm:Shape}
We have for each $\epsilon>0$, 
\begin{align}
\mathbb{P}_{n,\alpha} \left[ \sup_{t\in{[0,1]}}\left|\frac{K_{b_{t}\left(n\right)}}{n/\alpha(n)} -t \right| >\epsilon \right]
 \to 0
 \text{ as } n\to\infty.
\end{align}
\end{thm}
%
When we choose $t=1$ in Theorem~\ref{thm:Shape}, then $b_t(n) = \alpha(n)$ and we 
recover the result in  \cite{BS17}. Furthermore, if we define 
\begin{align}
 \nu_\epsilon := \lim_{n\to\infty} \frac{K_{b_{\epsilon}\left(n\right)}}{K_{\alpha\left(n\right)}} 
 \ \text{ for } \ \epsilon>0
\end{align}
then $\nu_\epsilon$ can be interpreted as the limit as $n\to\infty$ of the fraction of cycles smaller than $b_\epsilon(n)$.
Theorem~\ref{thm:Shape} now shows that $\nu_\epsilon \to \epsilon$ for all $0< \epsilon \leq 1$.
Since $b_\epsilon(n) = \alpha(n)\big(1+o(1) \big)$ for all $\epsilon>0$, 
we immediately get that almost all cycles live in a scale of the form $\alpha(n)\big(1+o(1) \big)$.

A theorem similar to Theorem~\ref{thm:Shape} can be proved for the number of indices. If we set 
$S_{m} = \sum_{j=1}^m j C_j$, then trivially $S_\alpha = n$, and 
we can show that 
\begin{align}
\mathbb{P}_{n,\alpha} \left[ \sup_{t\in{[0,1]}}\left|\frac{S_{b_{t}\left(n\right)}}{n} -t \right| >\epsilon \right] \label{eq:IShape}
 \to 0
 \text{ as } n\to\infty.
\end{align}
The proof, which is similar to the proof of Theorem~\ref{thm:Shape}, can be found in \cite[Theorem 2.7.2]{S18}.
In the next theorem we take a closer look at the fluctuations about the limit shape of $K_{b_t(n)}$.
%
\begin{thm}
\label{thm:Fluc}
Let
\begin{equation}\label{eq:FlukDef}
L_{t}\left(n\right):=\frac{K_{b_{t}\left(n\right)}-\sum_{j=1}^{b_{t}\left(n\right)}\frac{x^{j}_{n,\alpha}}{j}}{\sqrt{n/\alpha\left(n\right)}}.
\end{equation}
Then $\left(L_{t}\left(n\right)\right)_{t\in\left[0,1\right]}$ converges 
 in distribution to the standard 
 Brownian bridge in $\mathcal{D}\left[0,1\right]$, where $\mathcal{D}\left[0,1\right]$ is the space of cadlag functions on $[0,1]$, endowed with the Skorohod topology.
\end{thm} 
\begin{rem}$ $
(1)
    As above, we can do the same construction for the indices 
    instead of the cycles. With $S_m$ being as in the remark 
    after Theorem~\ref{thm:Shape}, we have that 
    \[
    \tilde{L}_{t}\left(n\right):=\frac{S_{b_{t}\left(n\right)}-\sum_{j=1}^{b_{t}\left(n\right)}x^{j}_{n,\alpha}}{\sqrt{n\alpha\left(n\right)}}
    \]
    converges to the Brownian bridge in $\mathcal{D}\left[0,1\right]$.
    The proof is similar to the one of Theorem 
    \ref{thm:Fluc}, so we refer to \cite[Theorem 2.7.6]{S18}.
\\
(2)
    When $t=1$ in Theorem \ref{thm:Fluc}, the 
    variance of the limit is zero. However, it has been shown in \cite{BS17} that 
    there exists a different rescaling so that 
    the Gaussian fluctuations persist in the limit: We have
    \begin{align}
    \frac{K_{\alpha(n)}-\sum_{j=1}^{\alpha(n)}\frac{x_{n,\alpha}^{j}}{j} }{\sqrt{\frac{n}{\alpha(n)\left(\log\left(n/\alpha(n)\right)\right)^2}}} \stackrel{d}{\longrightarrow} \caN(0,1). \label{eq:totalCLT}
    \end{align}
    Of course, no such statement can hold for $S_{\alpha(n)}$ since 
    $S_{\alpha(n)} - \sum_{j=1}^{\alpha(n)} x_{n,\alpha}^j =S_{\alpha(n)}-n= 0$.
\\
(3)
    For unrestricted permutations, Delaurentis and Pittel 
    \cite{DP85} show that 
    the stochastic process 
    \begin{align}
    \label{delaurentispittel}
    \left(\frac{\sum_{j=1}^{\left\lfloor n^{t}\right\rfloor }C_{j}-t\log\left(n\right)}{\sqrt{\log\left(n\right)}}\right)_{t\in\left[0,1\right]}
    \end{align}
    converges in distribution to the Brownian motion in $\left[0,1\right]$. 
    Interestingly, this holds for restricted permutations as well, 
    and we have already shown it! Indeed, the convergence 
    in total variation distance from Theorem~\ref{main thm 2}
    is strong enough to show that for all $t < a_1$ (cf.\ 
    \eqref{condition on alpha}), 
    convergence to the Brownian motion 
    also holds when the $C_j$ in \eqref{delaurentispittel} are 
    those of constrained permutations. 
    Hence, in the 
    case of constrained permutations, we actually have two functional 
    central limit theorems: one for 'short' cycles and one 
    for the ones very close to the maximal cycle length.\\
(4) The asymptotic behaviour of the longest cycles in a random permutation is one of the most frequently asked questions
and is in particular still open for random permutations with polynomially and logarithmically growing cycle weights.
We denote by $\ell_1(\sigma)$ the length of the longest cycle in a permutation, $\ell_2(\sigma)$ the length of the second longest cycle in a permutation and so on.
%
We have for each $k\in\N$
\begin{align}
\frac{1}{\alpha(n)} (\ell_1,\ell_2,\ldots,\ell_k)
\stackrel{d}{\longrightarrow}
(\underbrace{1,1,\ldots,1}_{\text{k times}}).
\end{align}
Further, if $\alpha(n) = \mathcal{O}(n^{1/2})$ and $\alpha\left(n\right) \geq n^{\frac{1}{7}+\delta}$ for some $\delta>0$ then
\begin{align}
 \mathbb{P}_{n,\alpha} \left[(\ell_1,\ell_2,\ldots,\ell_k) \neq \bigl(\alpha(n),\ldots,\alpha(n)\bigr) \right]
 \to 0 
 \ \text{ as }n\to\infty.
 \label{eq:longest_cycles_small_alpha}
\end{align}
%
These statements follow immediately from Theorems~\ref{main thm 3} and~\ref{thm:Shape}.
\end{rem}

\section{Generating functions and the saddle-point method}
\label{sec:Generating}
Generating functions and their connection with analytic combinatorics form the backbone of the proofs in this paper. 
More precisely, we will determine formal generating functions for all relevant moment-generating functions and then 
use the saddle-point method to determine the asymptotic behaviour of these moment-generating functions as $n\to\infty$.

Let $\left(a_{n}\right)_{n\in\mathbb{N}}$ be a sequence of complex numbers. Then its ordinary generating function is defined as the formal power series
\[
f\left(z\right):=\sum_{n=1}^{\infty}a_{n}z^{n}.
\]
The sequence may be recovered by formally extracting the coefficients
\[
\left[z^n\right]f\left(z\right):=a_{n}
\]
for any $n$. The first step is now to consider a special case of P{\'o}lya's Enumeration Theorem, see \cite[\S 16, p.\:17]{Po37}, which connects permutations with a specific generating function.
\begin{lem}
Let $(q_j)_{j\in\N}$ be a sequence of complex numbers. 
We then have  the following identity between formal power series in $z$,
\begin{equation}\label{eq:symm_fkt}
\exp\left(\sum_{j=1}^{\infty}\frac{q_j z^j}{j}\right)
=\sum_{k=0}^\infty\frac{z^k}{k!}\sum_{\sigma\in S_k}\prod_{j=1}^{k}
q_{j}^{C_j},
\end{equation}
where $C_j=C_j(\sigma)$ are the cycle counts. If either of the
series in \eqref{eq:symm_fkt} is absolutely convergent, then so is
the other one.
\end{lem}
Extracting the $n$th coefficient yields
\begin{equation}
\label{relation to perms}	
\left[z^n\right]\exp\left(\sum_{j=1}^{\infty}\frac{q_j z^j}{j}\right) = \frac{1}{n!}\sum_{\sigma\in S_n}\prod_{j=1}^{n}q_{j}^{C_j}.
\end{equation}
%
Setting $q_j =\mathbbm{1}_{\left\{j\leq \alpha(n)\right\}}$
we obtain 
\begin{align}\label{eq:cNorm}
Z_{n,\alpha} := \frac{|S_{n,\alpha}|}{n!} = 
\left[z^n\right]\exp\left(\sum_{j=1}^{\alpha}\frac{z^j}{j}
\right).
\end{align}
For distinct numbers $1\leq m_k\leq \alpha(n),$ $1\leq k\leq K$ and $s_1,...,s_K\in\mathbb{R}$, we obtain
\begin{equation}
\label{eq: finite joint generating functions}
\mathbb{E}_{n,\alpha}\left[ \mathrm{e}^{\sum_{k=1}^{K}s_{k}C_{m_k} }   \right]
=
\frac{1}{Z_{n,\alpha}} \,\, \left[z^n\right]\exp\left(\sum_{k=1}^{K}(\mathrm{e}^{s_k}-1)\frac{z^{m_k}}{m_k}\right)\exp\left(\sum_{j=1}^{\alpha(n)}\frac{z^{j}}{j}\right).
\end{equation}
Similarly, for $0=t_0\leq t_1 <\dots <t_m \leq t_{m+1}=1$, we have
\begin{equation}
\mathbb{E}_{n,\alpha}\left[  \mathrm{e}^{\sum_{i=1}^{m}s_{i} K_{b_{t_i}(n)} }   \right] 
= 
\frac{1}{Z_{n,\alpha}}[z^n]\exp\left(\sum_{i=0}^{m}\sum_{j=b_{t_{i}}\left(n\right)+1}^{b_{t_{i+1}}\left(n\right)}\frac{\mathrm{e}^{\sum_{\ell=i+1}^{m}s_{\ell}}z^{j}}{j}\right).
\label{eq:joint_limitshape}
\end{equation}
At this stage, all parameters can depend on the system size $n$. The way to extract the series coefficients from expressions 
such as \eqref{eq: finite joint generating functions} and 
\eqref{eq:joint_limitshape} is the saddle point method,  
a standard tool in asymptotic analysis. 
The basic idea is to rewrite the 
expression \eqref{relation to perms} as a complex contour integral and
choose the path of integration in a convenient way. 
The details of this procedure depend on the situation at hand 
and need to be done on a case by case basis. 
A general overview over the saddle-point method can be found 
in \cite[page~551]{Flajolet2009}. 

We now treat the most general case of the 
saddle point method that is relevant for the present situation. 
Let $\bsq = (q_{j,n})_{1 \leq j \leq \alpha(n), n \in 
\bbN}$ be a triangular array. We assume that all 
$q_{j,n}$ 
are nonnegative and define $x_{n,\bsq}$ as the unique positive solution of
\begin{align}
n = \sum_{j=1}^{\alpha(n)} q_{j,n} x_{n,\bsq}^j.\label{eq:GenSaddle}
\end{align} 
Let further
\[
\lambda_{p,n}:=\lambda_{p,n,\alpha,\boldsymbol{q}} :=\sum_{j=1}^{\alpha(n)} q_{j,n}j^{p-1}x_{n,\bsq}^j,
\]
where $p$ is a natural number. Due
to Equation (\ref{eq:GenSaddle}),
\begin{equation}
\lambda_{p,n}\leq n\left(\alpha\left(n\right)\right)^{p-1}\label{eq:LambdaP}
\end{equation}
holds for all $p \geq 1$. 

Let us write $a_n\approx b_n$ when there exist constants $c_1,c_2>0$ such that
$$
 c_1 b_n \leq a_n \leq c_2 b_n
$$
for large $n$. We further say that
$$
f_n(t)=\mathcal{O}\left(g_n(t)\right)\text{ uniformly in }t\in T_n
$$
if there are constants $c,N>0$ such that
$
\sup_{t\in T_n}\left\{\left|\frac{f_n(t)}{g_n(t)}\right|\right\}\leq c
$
for all $n\geq N$.

We will call an 
array $\bsq$ {\em admissible} 
if the following three conditions are met:\\[1mm]
(i): We have 
\begin{align}\label{eq:propvor}
\alpha(n) \log x_{n,\boldsymbol{q}} \approx  \log \frac{n}{\alpha(n)}.
\end{align}
(ii): We have 
\begin{align}\label{eq:admisible}
\lambda_{2,n} \approx n\alpha(n).
\end{align}
(iii): There exists a sequence $n \mapsto b(n)$ with $b(n) / \alpha(n) < (1 - \delta)$ for some $\delta > 0$, and such that
$q_{j,n} \geq c > 0$ for all
$j \geq b(n)$ and some constant $c > 0$. \\[1mm]
Note that condition (i) implies in particular that 
$\lim_{n \to \infty} x_{n,\bsq} = 1$. \\
Let $B_r(0)$ denote the circle with center $0$ and 
radius $r$ in the complex plane. 
We will call a sequence of complex-valued 
functions $f_n$ {\em admissible} if the 
following three conditions are met:\\
(i): There exists $\delta > 0$ such that $f_n$ is holomorphic
on $B_{x_{n,\bsq} + \delta}(0)$ for all $n$. \\
(ii): 
There exist $K,N > 0$ so that for all $n\geq N$ we have 
\begin{equation} 
	\label{admiss funct 1}
\sup_{z\in\partial B_{x_{n,\boldsymbol{q}}}\left(0\right)} \left|f_{n}\left(z\right)\right|\leq n^K
\left|f_{n}\left(x_{n,\boldsymbol{q}}\right)\right|.
\end{equation}
(iii) Let
\begin{equation}\label{def theta_n}
\theta_n := n^{-\frac{5}{12}}(\alpha\left(n\right))^{-\frac{7}{12}}.
\end{equation} For 
\begin{equation} \label{admiss funct 2}
|\!|\!| f_n |\!|\!|_n := \theta_n \sup_{|\theta| \leq \theta_n}
\frac{\left|f_{n}^{\prime}\left(x_{n,\bsq} \e{\ii \theta} \right)\right| }{\left|f_{n}\left(x_{n,\bsq}\right)\right|},
\end{equation}
we have $\lim_{n\to\infty} |\!|\!| f_n |\!|\!|_n = 0$. \\[1mm]
We are now in the position to formulate our general
saddle point result. 
\begin{prop}\label{prop:GenSaddle}
Let $\bsq$ be an admissible triangular array, and $(f_n)$
an admissible sequence of functions. Then, 
\[
\left[z^n\right] f_n(z) \exp\left(
\sum_{j=1}^{\alpha(n)}\frac{q_{j,n}}{j} z^{j} \right) = 
f_n(x_{n,\bsq}) \frac{\e{\lambda_{0,n}}}
{x_{n,\bsq}^n \sqrt{2\pi \lambda_{2,n}}} \left( 
1 + \caO \left(\frac{\alpha(n)}{n}\right) \right) \left(1 + \caO\left(|\!|\!| f_n |\!|\!|_n \right)\right).
\]
Here, the implicit constants in the error terms depend on $(f_n)_n$ only via $K,N$ in \eqref{admiss funct 1}.
\end{prop}
\begin{proof}
Cauchy's integral formula gives
\begin{equation} \label{eq:cauchy} 
M_n := \left[z^n\right] f_n(z) \exp\left(
\sum_{j=1}^{\alpha(n)}\frac{q_{j,n}}{j} z^{j} \right) = 
\frac{1}{2\pi\mathrm{i}}\int_{\partial B_{r}\left(0\right)}
f_{n}\left(z\right)\exp\left(\sum_{j=1}^{\alpha(n)}
\frac{q_{j,n}}{j}z^j\right)\frac{\mathrm{d}z}{z^{n+1}}
\end{equation}
for any $r$ such that $f_n$ is holomorphic on $B_r(0)$. 
Condition (i) on $f_n$ guarantees that we can take 
$r = x_{n,\bsq}$
. We then rewrite 
\begin{align*}
M_{n} = \frac{1}{2\pi x_{n,\boldsymbol{q}}^{n}}\int_{-\pi}^{\pi}f_n\left(x_{n,\boldsymbol{q}}\mathrm{e}^{\mathrm{i}\theta}\right)\exp\left(\sum_{j=1}^{\alpha\left(n\right)}\frac{q_{j,n}}{j}\left(x_{n,\boldsymbol{q}}\mathrm{e}^{\mathrm{i}\theta}\right)^{j}-\mathrm{i}n\theta\right)\mathrm{d}\theta.
\end{align*}
For the remainder of the proof, we will write $x$ instead of $x_{n,\boldsymbol{q}}$ and $\alpha$ instead of $\alpha(n)$ 
for lighter notation. We define 
\begin{equation}
g_{n}\left(\theta\right):=\sum_{j=1}^{\alpha\left(n\right)}q_{j,n}\frac{\mathrm{e}^{\mathrm{i}j\theta}-1}{j}x^{j}-\mathrm{i}n\theta\label{eq:gNdef}
\end{equation}
and obtain
\begin{align*}
M_{n}
&=
\frac{\exp\left(\sum_{j=1}^{\alpha\left(n\right)}\frac{q_{j,n}}{j}x^{j}\right)}{2\pi x^{n}}\int_{-\pi}^{\pi}f_n\left(x\mathrm{e}^{\mathrm{i}\theta}\right)\exp\left(g_{n}\left(\theta\right)\right)\mathrm{d}\theta.
\end{align*}
Note that 
 $g_n(0)=g^{\prime}_n(0) = 0$, $g^{(p)}_n(0) = \mathrm{i}^p
\lambda_{p,n}$, and $\left|g^{(p)}_n(\theta) \right| \leq 
\lambda_{p,n}$. 

For $|\theta| \leq \theta_n$ (see \eqref{def theta_n}), equation 
(\ref{eq:LambdaP}) implies that 
$\lambda_{p,n} |\theta|^p \leq  (n/\alpha)^{1 - 5p/12}$. 
Therefore a Taylor expansion around $0$ gives 
\begin{align*}
g_{n}(\theta) = -\frac{\lambda_{2,n}}{2}\theta^{2}
-\mathrm{i}\frac{\lambda_{3,n}}{6}\theta^3
+\caO \left(\lambda_{4,n}\theta^{4}\right)
\end{align*}
and 
\begin{align}\label{eq:gN}
\begin{split}
\exp (g_n(\theta)) & = \exp \left(-\tfrac{\lambda_{2,n}}{2} \theta^2 \right)
\Big(1-\mathrm{i}\frac{\lambda_{3,n}}{6}\theta^{3}+ \caO
\left(\lambda_{3,n}^{2}\theta^{6}\right)\Big) 
\Big( 1 + \caO(\lambda_{4,n} \theta^4) \Big),
\end{split}
\end{align}
where the error terms are uniform in $\theta\in [-\theta_n,\theta_n]$.
As for $f_n$, we have 
\[
f_n(x \e{\ii\theta}) = f_n(x) + \ii \int_0^\theta 
f'_n(x \e{\ii \varphi}) x \e{\ii \varphi} \, \dd \varphi.
\]
Estimating the modulus of the integrand in the second term by its maximum
and using assumption \eqref{admiss funct 2}, we find that, uniformly in  
$\theta\in [-\theta_n,\theta_n]$, 
\[
f_n(x \e{\ii \theta}) = f_n(x) 
\left(1 + \caO\left(|\!|\!| f_n |\!|\!|_n\right)\right).
\]
Here, the implicit constant in $\caO\left(|\!|\!| f_n |\!|\!|_n\right)$ is independent of $(f_n)_n$.
Putting things together, we have 
\begin{align*}
\int_{-\theta_n}^{\theta_n} f_n\left(x\mathrm{e}^{\mathrm{i}
\theta}\right)\exp\left(g_{n}\left(\theta\right)\right)
\mathrm{d}\theta 
= & f_n(x)
\int_{-\theta_n}^{\theta_n} \e{- \frac{\lambda_{2,n} \theta^2}{2}}
\left(1+\caO\left(\lambda_{3,n}^{2}\theta^{6} + \lambda_{4,n} \theta^4  \right) \right)
\, \dd \theta\\
&+f_n(x)
\int_{-\theta_n}^{\theta_n}\e{- \frac{\lambda_{2,n} \theta^2}{2}} 
\caO\left(|\!|\!| f_n |\!|\!|_n\right)
\dd \theta.
\end{align*}
By \eqref{eq:admisible}, $\lambda_{2,n} \theta_n^2 \approx n^{1/6} \alpha^{-1/6}$, which diverges as $n \to \infty$. 
The standard estimate on Gaussian tails gives that for all 
$m \in \bbN$, 
\[
\int_{-\theta_n}^{\theta_n} \e{- \frac{\lambda_{2,n} \theta^2}{2}}\, \dd \theta = \int_{-\infty}^{\infty} \e{- \frac{\lambda_{2,n} \theta^2}{2}}\, \dd \theta + 
\caO ( \lambda_{2,n}^{-m}) = \frac{\sqrt{2\pi}}
{\sqrt{\lambda_{2,n}}} + \caO( \lambda_{2,n}^{-m}).
\]
A scaling argument, \eqref{eq:LambdaP} and assumption 
\eqref{eq:admisible} give  
\[
\int_{-\theta_n}^{\theta_n} \e{- \frac{\lambda_{2,n} \theta^2}{2}} \lambda_{3,n}^2 |\theta|^6 \, \dd \theta \leq 15
\frac{\sqrt{2\pi}}
{\sqrt{\lambda_{2,n}}} \frac{\lambda_{3,n}^2}
{\lambda_{2,n}^{3}} = \frac{\sqrt{2\pi}}
{\sqrt{\lambda_{2,n}}} \caO\left(\tfrac{\alpha}{n}\right)
\]
and
\[
\int_{-\theta_n}^{\theta_n} \e{- \frac{\lambda_{2,n} \theta^2}{2}} \lambda_{4,n} |\theta|^4 \, \dd \theta \leq 3
\frac{\sqrt{2\pi}}
{\sqrt{\lambda_{2,n}}} \frac{\lambda_{4,n}}
{\lambda_{2,n}^{2}} = \frac{\sqrt{2\pi}}
{\sqrt{\lambda_{2,n}}} \caO\left(\tfrac{\alpha}{n}\right).
\]
Altogether, we find that 
\[
\int_{-\theta_n}^{\theta_n} f_n\left(x\mathrm{e}^{\mathrm{i}
\theta}\right)\exp\left(g_{n}\left(\theta\right)\right)
\mathrm{d}\theta = f_n(x) 
\sqrt{\frac{2 \pi}{\lambda_{2,n}}} \left(1 + \caO\left( \tfrac{\alpha}{n}\right) \right) 
(1 + \caO(|\!|\!| f_n |\!|\!|_n).
\]
What remains to be shown is that
\begin{equation}
\label{outside area}
\int_{\left|\theta\right|\geq\theta_n}f_n\left(x\mathrm{e}^{\mathrm{i}
\theta}\right)\exp\left(g_{n}\left(\theta\right)\right)\mathrm{d}\theta=
\caO\left(f_n(x)
\frac{\alpha\left(n\right)}{n\sqrt{\lambda_{2,n}}}\right),
\end{equation}
where the implicit error term only depends on $(f_n)_n$ via $K,N$.
We have
$
- \Re g_n(\theta) = \sum_{j=1}^{\alpha} \frac{q_{j,n}}{j} 
( 1 - \cos(j \theta)) x^j .
$
For $\theta_n \leq \theta < \pi/\alpha$, due to $-\partial_\theta \Re g_n(\theta) > 0$, we have 
\begin{equation}\label{eq:Aussen1} 
- \Re g_n(\theta) \geq - \Re g_n(\theta_n) \approx \theta_n^2 \lambda_{2,n} \approx 
\Big( \frac{n}{\alpha} \Big)^{1/6}
\end{equation}
by assumption \eqref{eq:admisible}. 
For $\theta > \frac{\pi}{\alpha}$, let us first assume that $q_{j,n}\geq c>0$ for all $n$ and $j$,
i.e.\ $b(n) = 1$ in assumption (iii). 
We use that 
\[
- \Re g_n(\theta) = \sum_{j=1}^{\alpha} \frac{q_{j,n}}{j} 
( 1 - \cos(j \theta)) x^j 
\geq \frac{c}{\alpha} \sum_{j=1}^{\alpha} 
( 1 - \cos(j \theta)) x^j =: c r_n(\theta)
\]
and
\begin{equation}
\label{r_n}
\begin{split}
r_n(\theta) & = \frac{1}{\alpha} \Big( x \frac{x^{\alpha} - 1}{x-1} - \Re \Big( 
x \e{\ii \theta} \frac{x^{\alpha} \e{\ii \theta \alpha} - 1}
{x \e{\ii \theta} - 1} \Big) \Big)   
 \geq 
\frac{2}{\pi^{2}}\frac{x^{\alpha+1}}{\alpha\left(x-1\right)}
\frac{\theta^{2}}{\left(x-1\right)^{2}+\theta^{2}} -\frac{2x}{\alpha
\left(x-1\right)}.
\end{split}
\end{equation}
The calculations for the final inequality can e.g.\ be found in 
\cite[Lemma 12]{manstavivcius2016local}. By \eqref{eq:propvor}, there 
exist $c_1,c_2 > 0$ with 
$
c_1 \log \frac{n}{\alpha} \leq \alpha \log x 
\leq c_2 \log \frac{n}{\alpha}.
$
Thus $x \sim 1$, and $x-1 \sim \log x \geq \frac{c_1}{\alpha} 
\log \frac{n}{\alpha}$. 
So the second term on the right hand side of \eqref{r_n} 
converges to zero. 
For the first term, we use that $\theta^2/((x-1)^2 + \theta^2)$ 
is monotone increasing in 
$\theta$, and find an asymptotic lower bound of the form  
\begin{equation}\label{b zwischen}
\frac{2}{\pi^{2}}\frac{x^{\alpha+1}}{c_2 \log \frac{n}{\alpha}}
\frac{\pi^2 \alpha^{-2}}{c_2^2 \alpha^{-2} \big( \log \frac{n}{\alpha} 
\big)^2 +\pi^2 \alpha^{-2}} \sim \frac{2}{c_2^3} 
\frac{x^{\alpha+1}}{\big(\log \frac{n}{\alpha}\big)^3}.
\end{equation}
Since $x^{\alpha} \geq \big( \frac{n}{\alpha} \big)^{c_1}$, 
and using condition \eqref{admiss funct 1}, we conclude that when 
$\theta \geq \theta_n$ and $n\geq N$, 
$ | f_n(x \e{\ii \theta}) \e{g_n(\theta)} |\leq n^K \left|\e{g_n(\theta)}\right|$ vanishes faster than 
all powers of $1/n$.  This shows the claim in the case $b(n) = 1$. 
For the case of general 
$b(n)$, we have 
\begin{equation}
\label{general b}
\begin{split}
- \Re g_n(\theta) & \geq \frac{1}{\alpha} \sum_{j=1}^{\alpha} 
q_{j,n} (1 - \cos (\theta j)) x^j = c r_n(\theta) + 
\frac{1}{\alpha} \sum_{j=1}^{\alpha} (q_{j,n} - c) (1 - \cos (\theta j)) x^j \\
& \geq c r_n(\theta) - \frac{2 c}{\alpha}  \sum_{j=1}^{b(n)} x^j
\geq cr_n(\theta) \left( 1 - \frac{2b(n)}{r_n(\theta)\alpha} x^{b(n)} \right).
\end{split}
\end{equation}
By assumption, $b(n) / \alpha \leq 1 - \delta$ for some $\delta > 0$, and 
then 
$
x^{b(n) - \alpha} \leq \big( \tfrac{n}{\alpha} \big)^
{c_1 \frac{b(n)-\alpha}{\alpha}} \leq \big( \tfrac{n}{\alpha} \big)^{- c_1 
\delta}.
$
Thus, by applying \eqref{b zwischen}, the bracket on the right hand side of \eqref{general b} converges to 
$1$ as $n \to \infty$, and the proof is finished. 
%
\end{proof}

\section{Proofs of the main results}\label{sec:proofs}
We establish most of our results by computing moment generating functions. In the cases we consider, it is a consequence of \cite{Yakymiv2011} that pointwise convergence of the moment generating functions in the sector $(\mathbb{R}_0^+)^d$ is sufficient to establish convergence in distribution of $d$-dimensional random variables.
The first result shows that the triangular array $\bsq$ with 
$q_{j,n} = \mathbbm{1}_{\left\{j\leq\alpha(n)\right\}}$ is admissible. 
\begin{lem}\label{lem:StaSad}
Let $x_{n,\alpha}$ be defined by equation \eqref{eq:StaSad}. We have, as $n\to\infty$:
\begin{equation}
\alpha\left(n\right)\log\left(x_{n,\alpha}\right)
=
\log\left(\frac{n}{\alpha\left(n\right)}\log\left(\frac{n}{\alpha\left(n\right)}\right)\right)+\caO\left(\frac{\log\left(\log\left(n\right)\right)}{\log\left(n\right)}\right).
\label{eq:StaSadAs}
\end{equation}
In particular,
$x_{n,\alpha}\geq1, \ 
\lim_{n\rightarrow\infty}x_{n,\alpha}=1
\ \text{ and } \
x_{n,\alpha}^{\alpha\left(n\right)}\sim\frac{n}{\alpha\left(n\right)}\log\left(\frac{n}{\alpha\left(n\right)}\right)$.
Furthermore,
\begin{equation}\label{eq:lambda2alt}
\sum_{j=1}^{\alpha(n)} jx_{n,\alpha}^j \sim n\alpha(n).
\end{equation}
\end{lem}
The first part of the lemma is a reformulation of Lemma 4.11 in \cite{BS17}, which in turn 
follows \cite{manstavivcius2016local}. In the latter reference, the claims 
are actually shown for more general functions $\alpha$.
Equation \eqref{eq:lambda2alt} has been proved in Lemma 9 in \cite{manstavivcius2016local}. It may also be derived as a special case of Lemma \ref{lem:Lambda2}.

\subsection{Proof of Proposition \ref{mean}}
\label{sec:mean}
Equation \eqref{mu asymptotics} follows directly from Lemma \ref{lem:StaSad}. We apply equation \eqref{eq: finite joint generating functions} with $K=1$, differentiate with respect to $s_1$, set $s_1=0$ and obtain
\begin{align*}
\mathbb{E}_{n,\alpha}\left[C_{m(n)} \right]=\frac{1}{Z_{n,\alpha}}\left[z^n\right]\frac{z^{m(n)}}{m(n)}\exp\left(\sum_{j=1}^{\alpha(n)}\frac{z^j}{j}\right).
\end{align*}
We may now apply Proposition \ref{prop:GenSaddle} with $f_n(z)=\frac{z^{m(n)}}{m(n)}$ and $q_{j,n}=\mathbbm{1}_{\{j\leq\alpha(n)\}}$. The array $\boldsymbol{q}$ is admissible by Lemma \ref{lem:StaSad} and $m(n)\leq \alpha(n) =o\left(\theta_n^{-1}\right)$ shows admissibility of $(f_n)$.
The claim then follows from $\mathbb{E}_{n,\alpha}\left[C_{m(n)} \right] \sim f_n(x_{n,\alpha})$.

\subsection{Proof of Theorem \ref{main thm 2}}\label{sec:TVDconv}
We follow the ideas in \cite{AT92}, where the case of uniform permutations 
is treated. 
Let $\left(Z_{k}\right)_{k}$ be independent random variables with $Z_{k}\sim\mathrm{Poi}\left(\frac{1}{k}\right)$ for $k\in\mathbb{N}$ and let
\begin{equation}
T_{b_{1}b_{2}}:=\sum_{k=b_{1}+1}^{b_{2}}kZ_{k}.\label{eq:Tdef}
\end{equation}
Let $\boldsymbol{C}_b=\left(C_1,C_2,\dots, C_b\right)$ the vector of the cycle counts up to length $b$, $\boldsymbol{Z}_b =\left(Z_1, Z_2, \dots, Z_b \right)$, and $\boldsymbol{a}=\left(a_1,a_2,\dots, a_b\right)$ a vector. A corner stone for investigating the classical case of uniform random permutations is the so-called conditioning relation \cite[Equation (1.15)]{ABT02},
\begin{equation}
\mathbb{P}_{n}\left[\boldsymbol{C}_{b}=\boldsymbol{a}\right]=\mathbb{P}\left[\left.\boldsymbol{Z}_{b}=\boldsymbol{a}\right|T_{0n}=n\right].\label{eq:Poissondarstellung}
\end{equation}
Since
$\mathbb{P}_{n,\alpha}=\mathbb{P}_{n}\left[\left.\cdot\right|C_{\alpha\left(n\right)+1}=...=C_{n}=0\right]$,
an analogue of Equation \eqref{eq:Poissondarstellung} holds for $b\leq\alpha\left(n\right)$:
\begin{equation}
\mathbb{P}_{n,\alpha}\left[\boldsymbol{C}_{b}=\boldsymbol{a}\right]=\mathbb{P}\left[\left.\boldsymbol{Z}_{b}=\boldsymbol{a}\right|T_{0\alpha\left(n\right)}=n\right].\label{eq:Poissondarst}
\end{equation}
Let $L(\boldsymbol{a}):=\sum_{k=1}^{b\left(n\right)}ka_{k}$. For $\boldsymbol{a}\in\mathbb{N}^{b\left(n\right)}$ with 
$L(\boldsymbol{a})=r$, independence of the $Z_k$ gives
\begin{align*}
\mathbb{P}\left[\left.\boldsymbol{Z}_{b\left(n\right)}=\boldsymbol{a}\right|T_{0\alpha\left(n\right)}=n\right]
= \frac{\mathbb{P}\left[\boldsymbol{Z}_{b\left(n\right)}=\boldsymbol{a}\right]\mathbb{P}\left[T_{b\left(n\right)\alpha\left(n\right)}=n-r\right]}{\mathbb{P}\left[T_{0\alpha\left(n\right)}=n\right]}.
\end{align*}
Define $\bbP_{n,b(n),\alpha}$ and $\tilde \bbP_{b(n)}$ as in Theorem \ref{main thm 2}, and let $d_{b(n)} := 
\| \bbP_{n,b(n),\alpha} - \tilde \bbP_{b(n)} \|_{\rm TV}$. By
\eqref{eq:Poissondarstellung},
\begin{align*}
d_{b\left(n\right)}= &
 \sum_{r=0}^{\infty} \sum_{\boldsymbol{a}:L\left(\boldsymbol{a}\right)=r}\mathbb{P}\left[\boldsymbol{Z}_{b\left(n\right)}=\boldsymbol{a}\right]\left(1-\frac{\mathbb{P}\left[T_{b\left(n\right)\alpha\left(n\right)}=n-r\right]}{\mathbb{P}\left[T_{0\alpha\left(n\right)}=n\right]}\right)_+\\
= & \sum_{r=0}^{\infty}\mathbb{P}\left[T_{0b\left(n\right)}=r\right]\left(1-\frac{\mathbb{P}\left[T_{b\left(n\right)\alpha\left(n\right)}=n-r\right]}{\mathbb{P}\left[T_{0\alpha\left(n\right)}=n\right]}\right)_+ \\ 
 \leq & 
\mathbb{P}\left[T_{0b\left(n\right)}\geq \rho_{n}b\left( n\right)+1\right]+\sum_{r=0}^{\rho_{n}b\left(n\right)}\mathbb{P}\left[T_{0b\left(n\right)}=r\right]\left(1-\frac{\mathbb{P}\left[T_{b\left(n\right)\alpha\left(n\right)}=n-r\right]}{\mathbb{P}\left[T_{0\alpha\left(n\right)}=n\right]}\right)_+,
\end{align*}
where $\rho_n > 0$ is arbitrary for now. In 
\cite[Lemma 8]{AT92} it is shown that 
\[
\mathbb{P}\left[T_{0b\left(n\right)}\geq \rho_n b\left(n\right)\right]\leq
\left(\frac{\rho_n}{\mathrm{e}}\right)^{-\rho_n}.
\]
So $\mathbb{P}\left[T_{0b\left(n\right)}\geq  \log (n) b\left(n\right) \right]$ decays faster than any power of $n$.
The proof is then concluded by plugging $\rho_n = \log n$ into the 
estimate of the lemma below. 
\begin{lem}{\label{lem:rKlein}}
Let $b\left(n\right)=o\left(\frac{\alpha\left(n\right)}{\log\left(n\right)}\right)$
and $\rho_{n}=\caO\left(\log\left(n\right)\right)$. Then,
\[
\max_{1\leq r\leq \rho_{n}b\left(n\right)}\left(1-\frac{\mathbb{P}\left[T_{b\left(n\right)\alpha\left(n\right)}=n-r\right]}{\mathbb{P}\left[T_{0\alpha\left(n\right)}=n\right]}\right)_{+}=\caO\left(\frac{\alpha(n)}{n} + \frac{b(n)}{\alpha(n)} \log(n)\right)
\]
as $n\rightarrow\infty$.
\end{lem}
\begin{proof}
We have
$
\mathbb{E}[z^{T_{b_{1}b_{2}}}]= \exp\left(\sum_{j=b_{1}+1}^{b_{2}}\frac{z^{j}-1}{j}\right)
$
. Therefore, 
\begin{equation}\label{4.4a}
\bbP [ T_{b(n)\alpha(n)} = n-r ] = 
[z^{n-r}] \e{ \sum_{j=b(n)+1}^{\alpha(n)} \frac{z^j-1}{j}} = \e{ - \sum_{j=b(n)+1}^{\alpha(n)} \frac{1}{j}}
[z^{n}] z^{r} \e{ \sum_{j=b(n)+1}^{\alpha(n)} \frac{z^j}{j}}  
\end{equation}
and 
\begin{equation}\label{4.4b}
\bbP [ T_{0\alpha(n)} = n ] = \e{ - \sum_{j=b(n)+1}^{\alpha(n)} \frac{1}{j}}
[z^{n}] \e{ \sum_{j=1}^{b(n)} \frac{z^j-1}{j}} \e{ \sum_{j=b(n)+1}^{\alpha(n)} \frac{z^j}{j}}. 
\end{equation}
Since the factors of $ \exp\left( - \sum_{j=b(n)+1}^{\alpha(n)} \frac{1}{j}\right)$ will cancel in the quotient of the two terms, we see that we are in the situation of Proposition \ref{prop:GenSaddle}. We have 
$q_{j,n}=\mathbbm{1}_{\{b(n)<j\leq\alpha(n)\}}$
in both \eqref{4.4a} and \eqref{4.4b}. Thus, the relevant saddle point
$x_{n,b,\alpha}$ is the unique
positive solution of
$
n=\sum_{j=b\left(n\right)+1}^{\alpha\left(n\right)}x_{n,b,\alpha}^{j}.
$
With $x_{n,\alpha}:=x_{n,0,\alpha}$ defined by \eqref{eq:StaSad}, we easily see that
$x_{n,\alpha}\leq x_{n,b,\alpha}\leq x_{n,\frac{\alpha}{2}}$
for large $n$. So Lemma \ref{lem:StaSad} shows
$\alpha \log x_{n,b,\alpha} \approx \log \frac{n}{\alpha(n)}$ and 
$\lambda_{2,n} \approx n \alpha(n)$. Thus $\bsq$ is admissible. \\
In \eqref{4.4a}, we have $f_n(z)=f^{(r)}(z)=z^r$ for all $n$ in the context of Proposition \ref{prop:GenSaddle}. Then, $f^{(r)}$ fulfils \eqref{admiss funct 1}
with $N=K=1$ for all $r\in\mathbb{N}$, and $|\!|\!|f^{(r)}|\!|\!|_n\leq r\theta_n=\mathcal{O}(\theta_n b(n)\log(n))$ uniformly in $r\leq \rho_n b(n)$. By the assumption on $(b(n))$, $f^{(r)}$ is admissible. \\
In \eqref{4.4b}, $f_n(z)=f_{b,n}(z)=\exp\left(\sum_{j=1}^{b(n)}\frac{z^j-1}{j}\right)$. We have
%
$
|\!|\!| f_{b,n} |\!|\!|_n \leq \theta_n 
\sum_{j=0}^{b(n)-1} x_{n,b,\alpha}^j \leq 
\theta_n b(n) x_{n,b,\alpha}^{b(n)}
$
and
\begin{equation}
\label{bn is bounded}
b(n) \log x_{n,b,\alpha} \approx \frac{b(n)}{\alpha(n)}  \log \left( \frac{n}{\alpha(n)} \right) = o(1)
\end{equation}
by the assumptions on $(b(n))$. Thus, $(f_{b,n})_n$ is admissible. We conclude
\begin{equation}\label{Stern2}
\frac{\mathbb{P}\left[T_{b\left(n\right)\alpha\left(n\right)}=n-r\right]}{\mathbb{P}\left[T_{0\alpha\left(n\right)}=n\right]}
= \frac{f_r(x_{n,b,\alpha})}{f_{b,n}(x_{n,b,\alpha})} 
\left(1 + \caO\Big( \frac{\alpha(n)}{n} +  \theta_n b(n) \log n\Big)\right),
\end{equation}
uniformly in $1\leq r \leq \rho_n b(n)$.
Now, $f^{(r)}(x_{n,b,\alpha}) \geq 1$ since $x_{n,b,\alpha}\geq 1$. On the other hand, writing $x$ instead of $x_{n,b,\alpha}$ and $f$ instead of $f_{b,n}$, we find 
\[
0 \leq \log (f(x)) = \sum_{j=1}^{b(n)} \frac{x^j-1}{j} = 
\int_1^x \sum_{j=0}^{b(n)-1} v^j  \, \dd v \leq (x-1) \, b(n) \, x^{b(n)}.
\]
By \eqref{bn is bounded}, $x^{b(n)} = \caO(1)$, and so $(x-1) b(n) x^{b(n)} = \caO\left( \frac{b(n)}{\alpha(n)} \log n \right)$.
We conclude
$1 \leq f(x) \leq 1 + \caO\left(\tfrac{b(n)}{\alpha(n)} \log n\right).$ Hence,
\[
\frac{f_r(x)}{f_{b,n}(x)} \geq 1 + \caO\left(\frac{b(n)}
{\alpha(n)} \log n\right),
\]
The claim now follows by inserting this into \eqref{Stern2}.
\end{proof}
\subsection{Proof of Theorem \ref{main thm 1}}\label{sec:counts}
%
%
%
%
%
%
%
Write $\mu_j := \mu_{m_j}$ and $\tilde{C}_{m_j} := C_{m_j}^{\left(\mu_j\right)}$. Let $s_j \geq 0$.
We have
\begin{align*}
& \mathbb{E}\left[\exp\left(\sum_{j=1}^k s_j \tilde{C}_{m_j}\right) \right] 
=  \sum_{l_1=0}^\infty \dots \sum_{l_k=0}^\infty \exp\left(\sum_{j=1}^k s_j l_j \right)\mathbb{P}\left[\tilde{C}_{m_1}=l_1, \dots , \tilde{C}_{m_k} = l_k \right]  \\
= &\frac{1}{Z}\sum_{l_1=0}^\infty \dots \sum_{l_k=0}^\infty \prod_{j=1}^k \frac{\exp(s_j l_j + \mu_j)}{\mu_j^{l_j}}\mathbb{P}_{n,\alpha}\left[C_{m_1}=l_1, \dots , C_{m_k} = l_k \right] \\
= & \frac{\exp\left(\sum_{j=1}^k \mu_j\right)}{Z}\sum_{l_1=0}^\infty \dots \sum_{l_k=0}^\infty \prod_{j=1}^k \exp[ l_j (s_j -\log \mu_j)]\mathbb{P}_{n,\alpha}\left[C_{m_1}=l_1, \dots , C_{m_k} = l_k \right] \\
= & \frac{\exp\left(\sum_{j=1}^k \mu_j\right)}{Z}  \mathbb{E}_{n,\alpha}\left[\exp\left(\sum_{j=1}^k \left(s_j - \log \mu_j\right) C_{m_j}\right) \right].
\end{align*}
Here, the normalization $Z$ depends on $n$.
By Equation \eqref{eq: finite joint generating functions},
the last expectation is given by $Z_{n,\alpha}^{-1} \left[z^n \right]f_n(z)\exp\left(\sum_{i=1}^{\alpha(n)}\frac{z^i}{i}\right)$, with
$f_n(z):=\exp\left(\sum_{j=1}^k  \left(\mathrm{e}^{s_j - \log \mu_j }-1\right)\frac{z^{m_j}}{m_j} \right)$. 
We are thus in the framework of Proposition \ref{prop:GenSaddle}, with $q_{j,n}=\mathbbm{1}_{\{j\leq \alpha(n)\}}$. By Lemma \ref{lem:StaSad}, it only remains to check admissibility of $(f_n)$.
For \eqref{admiss funct 1}, note that $|f_n(z)|\leq \exp\left(\sum_{j=1}^k |\mathrm{e}^{s_j-\log (\mu_j)}-1|\frac{x_{n,\alpha}^j}{j}\right)$ and
$$
\left|\mathrm{e}^{s_j-\log (\mu_j)}-1\right| = \left(\mathrm{e}^{s_j-\log (\mu_j)}-1\right)+2\left(1-\mathrm{e}^{s_j-\log (\mu_j)}\right)_+
\leq  \left(\mathrm{e}^{s_j-\log (\mu_j)}-1\right) +2.
$$
Since $\mu_j=\frac{x_{n,\alpha}^j}{m_j}$ by definition and $K_0 =\sup\{\mu_j:n\in\mathbb{N},j\leq k\}<\infty$ by assumption \eqref{limsup condition}, we get
\begin{equation}\label{Stern}
\left| f_n(z) \right| \leq K f_n(x_{n,\alpha})
\end{equation}
if $|z|=x_{n,\alpha}$, for all $s_k\geq 0$, with $K=\exp(2kK_0)$. For computing $|\!|\!|f_n|\!|\!|_n$, a direct calculation together with \eqref{Stern} gives
\begin{align*}
\left|\frac{f_n^\prime (z)}{f_n(x_{n,\alpha})}\right|\leq K \sum_{j=1}^k \left|\mathrm{e}^{s_j-\log (\mu_j)}-1\right| x_{n,\alpha}^{m_j-1}
\leq K \sum_{j=1}^k \left(\frac{\mathrm{e}^{s_j}}{\mu_j}+1\right)\mu_j m_j
\leq K\left(\sum_{j=1}^k \mathrm{e}^{s_j}+kK_0\right)\alpha.
\end{align*}
So, $|\!|\!|f_n|\!|\!|_n \leq  K\left(\sum_{j=1}^k \mathrm{e}^{s_j}+kK_0\right)\theta_n \alpha(n)=o(1)$, and $(f_n)$ is admissible.
By Proposition \ref{prop:GenSaddle}, we obtain
\begin{align*}
 \mathbb{E}\left[\exp\left(\sum_{j=1}^k s_j \tilde{C}_{m_j}\right) \right] \sim  \frac{\exp\left(\sum_{j=1}^k \mu_j\right)}{Z} f_n(x_{n,\alpha})
= \frac{\prod_{j=1}^k\exp\left(\mathrm{e}^{s_j}\right)}{Z}.
\end{align*}
By setting $s_j =0$ for all $j$, we may deduce $Z\to \mathrm{e}^k$ as $n\to\infty$, and the claim is proved.


\subsection{Proof of Theorem \ref{main thm 3}}

We now turn to the case of diverging expectation. The following proposition states the most general result in this regime.

\begin{prop}
\label{prop:CLTcycle}Let $m_{j}:\mathbb{N\rightarrow\mathbb{N}}$
for $1\leq j\leq k$ such that $m_{j}\left(n\right)\leq\alpha\left(n\right)$
	and $m_i(n)\neq m_j(n)$ for $i\neq j$.
Further, let $\mu_{m_j(n)}\left(n\right)$
as in \eqref{eq:def_mu_n}.
If $\mu_{m_j(n)}\left(n\right)\rightarrow\infty$ and $\theta_n \frac{x_{n,\alpha}^{m_{j}\left(n\right)}}{\sqrt{\mu_{m_j(n)}\left(n\right)}}\rightarrow0$
for all $j$, then
\[
\lim_{n\rightarrow\infty}\mathbb{E}_{n,\alpha}\left[\prod_{j=1}^{k}\exp\left(s_{j}\frac{C_{m_{j}\left(n\right)}-\mu_{m_j(n)}\left(n\right)}{\sqrt{\mu_{m_j(n)}\left(n\right)}}\right)\right]=\exp\left(\sum_{j=1}^{k}\frac{s_{j}^{2}}{2}\right)
\]
for all $s_{j}\geq0$.
\end{prop}
\begin{proof}
Write $\mu_j := \mu_{m_j(n)}(n)$. Applying equation \eqref{eq: finite joint generating functions} with $s_k$ replaced by $s_j /\sqrt{\mu_{j}}$, we are in the framework of
Proposition \ref{prop:GenSaddle}. Again $q_{i,n}:=\mathbbm{1}_{\{i\leq\alpha(n)\}}$, so $\boldsymbol{q}$ is admissible, and 
\[
f_{n}\left(z\right)=\exp\left[\sum_{j=1}^{k}\left(\exp\left(\frac{s_{j}}{\sqrt{\mu_{j}\left(n\right)}}\right)-1\right)\frac{z^{m_{j}\left(n\right)}}{m_{j}\left(n\right)}\right]\exp\left(-\sum_{j=1}^{k}s_{j}\sqrt{\mu_{j}\left(n\right)}\right)
\]
For admissibility of $(f_n)$, we compute
\begin{align*}
\sup_{z\in\partial B_{x_{n,\alpha}}\left(0\right)}\frac{\left|f_{n}^{\prime}\left(z\right)\right|}{\left|f_{n}\left(\left|z\right|\right)\right|}
\leq  \sum_{j=1}^{k}\left(\exp\left(\frac{s_{j}}{\sqrt{\mu_{j}\left(n\right)}}\right)-1\right)x_{n,\alpha}^{m_{j}\left(n\right)-1}.
\end{align*}
By our assumption on $\mu_{m_j}(n)$, $(f_n)$ is admissible and we may apply Proposition \ref{prop:GenSaddle}. 
Again the case $s_j=0$ for all $j$ deals with the normalizing constant, and so,
from
\begin{align*}
f_{n}\left(x_{n,\alpha}\right)= & \exp\left[\sum_{j=1}^{k}\left(\frac{s_{j}}{\sqrt{\mu_{j}\left(n\right)}}+\frac{s_{j}^{2}}{2\mu_{j}\left(n\right)}+\caO\left(\frac{s_{j}^{3}}{\left(\mu_{j}\left(n\right)\right)^{\frac{3}{2}}}\right)\right)\mu_{j}\left(n\right)-\sum_{j=1}^{k}s_{j}\sqrt{\mu_{j}\left(n\right)}\right]\\
= & \exp\left[\sum_{j=1}^{k}\frac{s_{j}^{2}}{2}\right]\left(1+\caO\left(\sum_{j=1}^{k}\frac{1}{\left(\mu_{j}\left(n\right)\right)^{\frac{1}{2}}}\right)\right)
\rightarrow  \exp\left[\sum_{j=1}^{k}\frac{s_{j}^{2}}{2}\right],
\end{align*}
we then conclude the claim.
\end{proof}
Since Lemma \ref{lem:StaSad} entails $\theta_n \frac{x_{n,\alpha}^{m_{j}\left(n\right)}}{\sqrt{\mu_{m_j(n)}\left(n\right)}}\rightarrow0$
for all $j$ if $a_1 >1/7$, Theorem \ref{main thm 3} follows.

%

\subsection{Proofs of Theorems \ref{thm:Shape} and \ref{thm:Fluc} and Equation \eqref{eq:IShape}}\label{sec:limitshape}
This section deals mainly with the proofs concerning the limit shape and fluctuations of cumulative cycle counts. 
We begin with equation \eqref{eq:joint_limitshape},
where we plug in $s_i/\gamma(n)$ instead of $s_i$
for a real-valued sequence
$(\gamma(n))_{n\in\mathbb{N}}$. 
In the terms of Proposition \ref{prop:GenSaddle}, this means that $f_n =1$ and $q_{j,n}=\mathrm{e}^{\sum_{l=i(j)}^{m}\frac{s_l}{\gamma(n)}}$ where $i(j):=\min\left\{1\leq l \leq m: b_{t_l}(n)\geq j\right\}$. Intuitively, any index $l$ with $b_{t_l}(n) \geq j$ contributes a factor of $\exp \left(s_l /\gamma(n) \right)$ to $q_{j,n}$ since the number of cycles of length $j$ is counted in $K_{b_{t_l}}(n)$ in this case. 
The saddle point of this problem is given by the unique positive solution 
$x_{n}\left(\boldsymbol{s}\right):=x_{n,\alpha,\gamma,\boldsymbol{t}}\left(\boldsymbol{s}\right)$
of
\begin{equation}
n=\sum_{i=0}^{m}\mathrm{e}^{\sum_{l=i+1}^{m}\frac{s_{l}}{\gamma(n)}}\sum_{j=b_{t_{i}}\left(n\right)+1}^{b_{t_{i+1}}\left(n\right)}\left(x_{n}\left(\boldsymbol{s}\right)\right)^{j}\label{eq:NewSaddleLimitshape}
.\end{equation}
Note that $x_{n}\left( \boldsymbol{0}\right)=x_{n,\alpha}$. 
Lemmata \ref{lem:NewSaddle} and \ref{lem:Lambda2} show that $\boldsymbol{q}$ is admissible and provide detailed information which will be useful for  investigating the moment generating function more closely.
\begin{lem}
\label{lem:NewSaddle}Let $\gamma\left(n\right)\to\infty$ with $\gamma\left(n\right)\geq \log(n)$ and $\boldsymbol{t}=\left(t_{i}\right)_{1\leq i\leq m}$
with $0=t_{0} \leq t_{1}<...<t_{i}<...<t_m\leq t_{m+1}=1$ and $s_{i}\geq0$ for
all $1\leq i\leq m$. Then
\begin{equation}
\alpha\left(n\right)\log\left(x_{n}\left(\boldsymbol{s}\right)\right)
=
\log\left(\frac{n}{\alpha\left(n\right)}\right) + \caO\left(\frac{\log\left(\log\left(n\right)\right)}{\log\left(n\right)}\right)
\label{eq:NewSaddleLimitshapeAs}
\end{equation}
locally uniformly in $\boldsymbol{s}$. In particular,
$
\lim_{n\rightarrow\infty}x_{n}\left(\boldsymbol{s}\right)=1
$
locally uniformly in $\boldsymbol{s}$.
\end{lem}
\begin{proof}
Let $\hat{x}_n(\bss)$ be the unique positive solution of $n\exp\left(-\sum_{i=1}^{m}\frac{s_i}{\gamma(n)}\right) = \sum_{j=1}^{\alpha(n)}(\hat{x}_{n}(\bss))^j$.
Since $s_{i}\geq0$ for all $i$, comparing equations (\ref{eq:StaSad})
and (\ref{eq:NewSaddleLimitshape}) yields
\begin{equation}
\hat{x}_{n}(\bss)     \leq x_{n}\left(\boldsymbol{s}\right)     \leq x_{n,\alpha(n)}.\label{eq:NewSadIneq}
\end{equation}
By a slightly more general version of Lemma~\ref{lem:StaSad} (cf. \cite[Lemma 9]{manstavivcius2016local}), we also have
\begin{align}
\alpha(n)\log\left(\hat{x}_n(\bss)\right)
&=
\log\left(\frac{n\exp\left(-\sum_{i=1}^{m}\frac{s_i}{\gamma(n)}\right)}{\alpha(n)}\right)+ \caO\left(\frac{\log\left(\log\left(n\right)\right)}{\log\left(n\right)}\right)\nonumber\\
&= \log\left(\frac{n}{\alpha(n)}\right) + \caO\left(\frac{\log\left(\log\left(n\right)\right)}{\log\left(n\right)}\right)
\label{eq:xHutAs}
\end{align}
locally uniformly in $\bss$ due to $\gamma(n)\to\infty$.
Equation \eqref{eq:NewSaddleLimitshapeAs} then follows from \eqref{eq:NewSadIneq} together with Lemma \ref{lem:StaSad} and equation \eqref{eq:xHutAs}.
\end{proof}
\begin{lem}
\label{lem:Lambda2}Let $\gamma\left(n\right)\to\infty$ with $\gamma\left(n\right)\geq \log(n)$ and $\boldsymbol{t}=\left(t_{1},...,t_{m}\right)^{T}$
with $0\leq t_{1}<...<t_{m}\leq1$ for $m\in\mathbb{N}$. Then, locally
uniformly in $\boldsymbol{s}=\left(s_{1},...,s_{m}\right)^{T}\in\left[0,\infty\right)^{m}$,
\[
\lambda_{2,n}=n\alpha\left(n\right)+\caO\left(\frac{n\alpha\left(n\right)}{\log\left(n\right)}\right).
\]
\end{lem}
\begin{proof}
W.l.o.g., let $0<t_{1}<1$ and $m=1$. As the following calculations
will show, larger values of $m$ pose no particular problem since
they only produce additional terms of similar structure and $b_{t_{k}}\left(n\right)\sim\alpha\left(n\right)$
for all $k\geq 1$ in this case. Moreover, let $x:=x_{n,\alpha,\gamma,\boldsymbol{t}}\left(\boldsymbol{s}\right)$.
Then, using that $\gamma\left(n\right)\geq \log(n)$, we obtain
\begin{align*}
\lambda_{2,n}
= & 
\mathrm{e}^{\frac{s_1}{\gamma(n)}}\sum_{j=1}^{b_{t_{1}}\left(n\right)}jx^{j}+\sum_{j=b_{t_{1}}\left(n\right)+1}^{\alpha\left(n\right)}jx^{j}
=  
\left(\sum_{j=1}^{\alpha(n)} jx^{j}\right)\left(1+\caO\left(\frac{1}{\log(n)}\right)\right)\\
=&
 \left(\alpha(n)\frac{x^{\alpha(n)+1}-1}{x-1} +\frac{x^{\alpha(n)+1}-1}{(x-1)^2}\right)\left(1+\caO\left(\frac{1}{\log(n)}\right)\right).
\end{align*}
Since $x\to1$ as $n\to \infty$, we have $x-1= \log(x) + \caO((x-1)^2)$. Using this together with Lemma~\ref{lem:NewSaddle} completes the proof.
\end{proof}
%
Having proved that $\boldsymbol{q}$ is admissible, Proposition \ref{prop:GenSaddle} yields, for $\gamma\left(n\right)\geq \log(n)$, $\boldsymbol{t}=\left(t_{1},...,t_{m}\right)^{T}$ and fixed $\boldsymbol{s}=\left(s_{1},...,s_{m}\right)^{T}\in\left[0,\infty\right)^{m}$, 
\[
M_{n,\gamma}\left(\boldsymbol{s}\right)
:=
\mathbb{E}_{n,\alpha}\left[\exp\left(\sum_{i=1}^{m}\frac{s_i}{\gamma(n)}K_{b_{t_{i}}\left(n\right)}\right)\right]
=
\frac{1}{Z_{n,\alpha}}\frac{1}{\sqrt{2\pi n\alpha (n)}}\exp\left[h_{n}\left(\boldsymbol{s}\right)\right]\left(1+o\left(1\right)\right),
\]
where $Z_{n,\alpha}$ is the normalizing constant in \eqref{eq:cNorm} such that $M_{n,\gamma}(\boldsymbol{0})=1$ and 
\begin{align}
 h_{n}\left(\boldsymbol{s}\right)
:=
h_{n,\alpha,\gamma,\boldsymbol{t}}\left(\boldsymbol{s}\right):=\sum_{i=0}^{m}\mathrm{e}^{\sum_{l=i+1}^{m}\frac{s_l}{\gamma (n)}}\sum_{j=b_{t_{i}}\left(n\right)+1}^{b_{t_{i+1}}\left(n\right)}\frac{\left(x_{n,\alpha,\gamma,\boldsymbol{t}}\left(\boldsymbol{s}\right)\right)^{j}}{j}-n\log\left(x_{n,\alpha,\gamma,\boldsymbol{t}}\left(\boldsymbol{s}\right)\right).
\label{eq:h(s)}
\end{align}
The next step is to extract more information by investigating the functions $h_{n}$.
The proofs will rest on a Taylor expansion of $h_{n}$ about $\boldsymbol{0}$,
so we need expressions and asymptotics for the derivatives of $h_{n}$.
We will prove in Section \ref{sub:Beweise} for $\gamma(n)\geq \log (n)$:
\begin{itemize}
\item[(i)] $\boldsymbol{s}\mapsto h_n(\boldsymbol{s})$ is infinitely often differentiable,
\item[(ii)] $\partial_{s_i} h_n(\boldsymbol{0})=\frac{1}{\gamma(n)}\sum_{j=1}^{b_{t_i}(n)}\frac{x_{n,\alpha}^j}{j}=t_i \frac{n}{\gamma(n)\alpha(n)}\left( 1 + o(1) \right)$,
\item[(iii)] $\partial_{s_{i_2}}\partial_{s_{i_1}} h_n(\boldsymbol{0}) = t_{i_2} ( 1 - t_{i_1} ) \frac{n}{(\gamma(n))^2\alpha(n)}\left( 1 + o(1) \right)$ for $i_2 \leq i_1$,
\item[(iv)] $\partial_{s_{i_2}}\partial_{s_{i_1}} h_n(\boldsymbol{s}) = \caO \left(\frac{n}{(\gamma(n))^2\alpha(n)} \right)$ locally uniformly in $\boldsymbol{s}$,
\item[(v)] $\partial_{s_{i_3}}\partial_{s_{i_2}} \partial_{s_{i_1}} h_n(\boldsymbol{s})= \caO \left(\frac{n}{(\gamma(n))^3\alpha(n)} \right)$  locally uniformly in $\boldsymbol{s}$.
\end{itemize}
Due to $M_{n,\gamma}(\boldsymbol{0}) = 1$, for fixed $\bss$ we therefore arrive at
\begin{equation}\label{eq:ShapeZ}
M_{n,\gamma}(\boldsymbol{s}) = \exp\left(\nabla h_n(\boldsymbol{0}) \cdot \boldsymbol{s} + \caO\left(\frac{n}{\gamma^2\alpha} \left| \boldsymbol{s} \right|^2 \right)\right) \left(1 + o(1)\right)
\end{equation}
and
\begin{equation}\label{eq:FlukZ}
M_{n,\gamma}(\boldsymbol{s}) = \exp\left(\nabla h_n(\boldsymbol{0}) \cdot \boldsymbol{s} + \frac{1}{2} \left\langle \boldsymbol{s}, H_{h_n}(\boldsymbol{0}) \boldsymbol{s} \right\rangle + \caO\left(\frac{n}{\gamma^3\alpha} \left| \boldsymbol{s} \right|^3 \right)\right) \left(1 + o(1)\right).
\end{equation}
So, by equation \eqref{eq:ShapeZ},
\begin{equation}\label{eq:Shape}
\lim_{n\to\infty}\mathbb{E}_{n,\alpha}\left[\exp\left(\sum_{i=1}^{m}\frac{s_{i}}{n/\alpha(n)}K_{b_{t_{i}}\left(n\right)}\right)\right]
=\lim_{n\to\infty}M_{n,\frac{n}{\alpha(n)}}(\boldsymbol{s})
=\exp\left(\sum_{i=1}^{m}s_{i}t_{i}\right),
\end{equation}
and, by equation \eqref{eq:FlukZ},
\begin{align}\label{eq:Fluk}
& \lim_{n\to\infty}\mathbb{E}_{n,\alpha}\left[\exp\left(\sum_{i=1}^{m}\frac{s_{i}}{\sqrt{n/\alpha(n)}}\left(K_{b_{t_{i}}\left(n\right)} - \sum_{j=1}^{b_{t_i}(n)}\frac{x_{n,\alpha}^j}{j}   \right)\right)\right] \\
= & \lim_{n\to\infty}M_{n,\sqrt{n/\alpha(n)}}(\boldsymbol{s}) \exp\left(- \nabla h_n(\boldsymbol{0}) \cdot \boldsymbol{s}   \right)
=\exp\left(\frac{1}{2}\left<\boldsymbol{s},A\left(\boldsymbol{t}\right)\boldsymbol{s}\right>\right),\nonumber
\end{align}
where $A\left(\boldsymbol{t}\right)=\left(A_{i_{1},i_{2}}\right)$
is symmetric with
$
A_{i_{1},i_{2}}=t_{i_{2}}\left(1-t_{i_{1}}\right)
$
for $i_{2}\leq i_{1}$. Note that $A(\boldsymbol{t})$ is the covariance matrix of the Brownian bridge.
We can now give the
\begin{proof}[Proof of Theorem \ref{thm:Shape}]
We apply arguments of the proof of Corollary 3.4 in \cite{CiZe13}. Let $\epsilon >0$ and choose $0=t_0 < t_1 <...<t_l =1$ such that $t_{j+1}-t_j < \frac{\epsilon}{2}$. Then, due to monotonicity,
$
\left|\frac{K_{b_t\left(n\right)}}{n/\alpha(n)}-t\right|>\epsilon
$
for some $t\in[0,1]$ implies the existence of an index $j$ such that
$
\left|\frac{K_{b_{t_j}\left(n\right)}}{n/\alpha(n)}-t_j\right|>\frac{\epsilon}{2}.
$
Then,
\begin{equation}\label{eq:Shapebeweis}
\mathbb{P}_{n,\alpha}\left[\sup_{t\in [0,1]}\left| \frac{K_{b_t\left(n\right)}}{n/\alpha(n)}-t\right|>\epsilon \right]
\leq \sum_{j=1}^l\mathbb{P}_{n,\alpha}\left[\left|\frac{K_{b_{t_j}\left(n\right)}}{n/\alpha(n)}-t_j\right|>\frac{\epsilon}{2}\right]
\xrightarrow{n\to\infty} 0
\end{equation}
by equations \eqref{eq:Shape} and \eqref{eq:Fluk}.
\end{proof}
Equation \eqref{eq:Fluk} establishes the convergence
of the finite-dimensional distributions of the fluctuations to those
of the Brownian bridge. In order to show that, under $\mathbb{P}_{n,\alpha}$,
the fluctuations $\left(L_{t}\left(n\right)\right)_{t\in\left[0,1\right]}$ defined in \eqref{eq:FlukDef}
converge as a process to the Brownian bridge, we also have to prove
tightness. 
We will apply the criterion that there are $N\in\mathbb{N},c>0$, and a nondecreasing continuous function $H$ on $[0,1]$ such that
\begin{equation}
\mathbb{E}_{n,\alpha}\left[\left|L_{t}\left(n\right)-L_{t_1}\left(n\right)\right|^{2}\left|L_{t_2}\left(n\right)-L_{t}\left(n\right)\right|^{2}\right] \leq c\left|H(t_{2})-H(t_{1})\right|^{2}\label{eq:tightnesscriterion}
\end{equation}
for all $0\leq t_{1}\leq t\leq t_{2}\leq 1$ and all $n\geq N$, which is an instance of \cite[Equation (13.14)]{Billingsley1999}.
\begin{prop}\label{prop:Straff}
The sequence of processes $\left(L_{t}\left(n\right)\right)_{t\in [0,1]}$
under $\mathbb{P}_{n,\alpha}$ is tight in $\mathcal{D}\left[0,1\right]$.
\end{prop}
In this paper we only prove tightness of $\left(L_{t}\left(n\right)\right)_{t\in [\delta,1]}$ for $0<\delta <1$ since the proof of the general case (in particular suitably generalizing Lemma \ref{lem:StraffHilfe} below) is very technical. The main reason for this is that one has to deal with the divergence of $(\log (t))^\prime=\frac{1}{t}$ as $t\to 0$ in the definition of $b_t(n)$. The proof of the general statement can be found in \cite[Theorem 2.7.5]{S18}.\\
The arguments in the proof of Proposition \ref{prop:Straff} further show that, by uniform integrability, we can substitute the mean $\mathbb{E}_{n,\alpha}[L_t(n)]$ for $\sum_{j=1}^{b_{t}(n)}\frac{x^{j}_{n,\alpha}}{j}$ in \eqref{eq:FlukDef} if we only consider convergence of the finite-dimensional distributions.
We are going to need the following
\begin{lem}\label{lem:StraffHilfe}
Let $0<\delta<1$. Then there are $N\in\mathbb{N}$ and $c>0$ such that
$$
\sum_{j=b_{t_1}(n)+1}^{b_{t_2}(n)} \frac{x_{n,\alpha}^j}{j} \leq c \frac{n}{\alpha(n)} (t_2-t_1)
$$
for all $n\geq N$ and $\delta\leq t_1 <t_2\leq 1$ satisfying $b_{t_2}(n)-b_{t_1}(n)\geq 2$.
\end{lem}
\begin{proof}
Let $N_1$ be such that $b_{\delta}(n)\geq \alpha(n)/2$ for all $n\geq N_1$. Then,
\begin{align*}
\sum_{j=b_{t_1}(n)+1}^{b_{t_2}(n)} \frac{x_{n,\alpha}^j}{j} &\leq \frac{2}{\alpha(n)} x_{n,\alpha}^{b_{t_1}(n)+1}\sum_{j=0}^{b_{t_2}(n)-b_{t_1}(n)-1}x_{n,\alpha}^j
\leq \frac{2}{\alpha(n)} x_{n,\alpha}^{\alpha(n)+1} \frac{x^{b_{t_2}(n)-b_{t_1}(n)}-1}{x_{n,\alpha} -1}\\
&\leq \frac{2}{\alpha(n)} x_{n,\alpha}^{\alpha(n)+1} \frac{\exp\left(\log(x_{n,\alpha})\left[\frac{\alpha(n)(\log(t_2)-\log(t_1))}{\log(n/\alpha(n))}+1\right]\right)-1}{x_{n,\alpha} -1}.
\end{align*}
By Lemma \ref{lem:StaSad}, $\log(x_{n,\alpha})\frac{\alpha(n)}{\log(n/\alpha(n))}\to 1$ as $n\to\infty$. Moreover, $\log(t_2)-\log(t_1)\leq \delta^{-1}(t_2-t_1)\leq \delta^{-1}$ and $\alpha(n) (\log(n/\alpha(n)))^{-1}(\log(t_2)-\log(t_1))\geq 1$ by assumption. Hence, there are $N\geq N_1, c_1>0$ such that
\begin{align*}
\exp\left(\log(x_{n,\alpha})\left[\frac{\alpha(n)(\log(t_2)-\log(t_1))}{\log(n/\alpha(n))}+1\right]\right) &\leq 1+c_1(\log(t_2)-\log(t_1))\
\leq 1+ \frac{c_1}{\delta} (t_2 - t_1).
\end{align*}
Thus, by Lemma \ref{lem:StaSad},
$$
\sum_{j=b_{t_1}(n)+1}^{b_{t_2}(n)} \frac{x_{n,\alpha}^j}{j} \leq \frac{2c_1}{\delta}\frac{1}{\alpha(n)} \frac{x_{n,\alpha}^{\alpha(n)+1}}{x_{n,\alpha}-1} (t_2 -t_1)
\leq c \frac{n}{\alpha(n)}(t_2 - t_1)
$$ 
for some $c>0$.
\end{proof}
\begin{proof}[Proof of Proposition \ref{prop:Straff}]
We prove equation (\ref{eq:tightnesscriterion}) with $H=\mathrm{id}$ for $\delta\leq t_1\leq t\leq t_2$. By definition,
\begin{align}
  I_{n}\nonumber &
=  \mathbb{E}_{n,\alpha}\left[\left|L_{t}\left(n\right)-L_{t_1}\left(n\right)\right|^{2}\left|L_{t_2}\left(n\right)-L_{t}\left(n\right)\right|^{2}\right]\nonumber\\
&=  \mathbb{E}_{n,\alpha}\left[\left(\frac{K_{b_{t}\left(n\right)}-K_{b_{t_{1}}\left(n\right)}-\sum_{j=b_{t_{1}}\left(n\right)+1}^{b_{t}\left(n\right)}\frac{x_{n,\alpha}^{j}}{j}}{\sqrt{n/\alpha\left(n\right)}}\right)^{2}\left(\frac{K_{b_{t_{2}}\left(n\right)}-K_{b_{t}\left(n\right)}-\sum_{j=b_{t}\left(n\right)+1}^{b_{t_{2}}\left(n\right)}\frac{x_{n,\alpha}^{j}}{j}}{\sqrt{n/\alpha\left(n\right)}}\right)^{2}\right].\label{eq:tightproduct}
\end{align}
We only have to deal with $t_1,t_2$ such that $b_{t_{2}}\left(n\right)-b_{t_{1}}\left(n\right)\geq 2$ because $I_n=0$ otherwise.
Consider the moment generating function
\begin{align*}
F_{n}\left(s_{1},s_{2}\right):= & \mathbb{E}_{n,\alpha}\left[\exp\left(s_{1}\frac{K_{b_{t}\left(n\right)}-K_{b_{t_{1}}\left(n\right)}}{\sqrt{n/\alpha\left(n\right)}}+s_{2}\frac{K_{b_{t_{2}}\left(n\right)}-K_{b_{t}\left(n\right)}}{\sqrt{n/\alpha(n)}}\right)\right]\\
= & \frac{1}{Z_{n,\alpha}}\left[z^{n}\right]\exp\left(\sum_{j=1}^{b_{t_{1}}\left(n\right)}\frac{z^{j}}{j}+\mathrm{e}^{\sqrt{\frac{\alpha(n)}{n}}s_1}\sum_{j=b_{t_{1}}\left(n\right)+1}^{b_{t}\left(n\right)}\frac{z^{j}}{j}+\mathrm{e}^{\sqrt{\frac{\alpha(n)}{n}}s_2}\sum_{j=b_{t}\left(n\right)+1}^{b_{t_{2}}\left(n\right)}\frac{z^{j}}{j}+\sum_{j=b_{t_{2}}\left(n\right)+1}^{\alpha\left(n\right)}\frac{z^{j}}{j}\right).
\end{align*}
Then $F_{n}$ is differentiable and
\begin{equation}
\mathbb{E}_{n,\alpha}\left[\left(\frac{K_{b_{t}\left(n\right)}-K_{b_{t_{1}}\left(n\right)}}{\sqrt{n/\alpha\left(n\right)}}\right)^{m_{1}}\left(\frac{K_{b_{t_{2}}\left(n\right)}-K_{b_{t}\left(n\right)}}{\sqrt{n/\alpha\left(n\right)}}\right)^{m_{2}}\right]=\left.\partial_{s_{1}}^{m_{1}}\partial_{s_{2}}^{m_{2}}F_{n}\left(s_{1},s_{2}\right)\right|_{\left(s_{1},s_{2}\right)=0}\label{eq:tightmoments}
\end{equation}
holds. By linearity of the expectation, we can expand the product
in equation (\ref{eq:tightproduct}) and then apply Equation (\ref{eq:tightmoments})
to each summand. A calculation then yields
\begin{align*}
I_{n}= & \frac{\alpha\left(n\right)^{2}}{Z_{n,\alpha}n^{2}}\left[z^{n}\right]\left[G_{n,t_{1},t}\left(z\right)G_{n,t,t_{2}}\left(z\right)\exp\left(\sum_{j=1}^{\alpha\left(n\right)}\frac{z^{j}}{j}\right)\right],
\end{align*}
where
\[
G_{n,t_{1},t}\left(z\right):=\left(\sum_{j=b_{t_{1}}\left(n\right)+1}^{b_{t}\left(n\right)}\frac{z^{j}-x_{n,\alpha}^{j}}{j}\right)^{2}+\sum_{j=b_{t_{1}}\left(n\right)+1}^{b_{t}\left(n\right)}\frac{z^{j}}{j}.
\]
The additional terms of the form $\sum_{j=b_{t_{1}}\left(n\right)+1}^{b_{t}\left(n\right)}\frac{z^{j}}{j}$
result from the product rule when calculating the second derivative
with respect to the same variable $s_{1}$. We now proceed as in the
proof of Proposition \ref{prop:GenSaddle} with $q_{j,n}=\mathbbm{1}_{\{j\leq\alpha(n)\}}$, which is admissible. The functions $G_{n,t_{1},t}\left(z\right)G_{n,t,t_{2}}\left(z\right)$ would play the role of $f_n$, but they only satisfy (i) and (ii) (by Lemma \ref{lem:partLambda0}). Since (iii) does in general not hold, we will have to make some adaptations. As in the proof of Proposition \ref{prop:GenSaddle}, by Cauchy's integral formula, we write $I_{n}$ as a
contour integral along $\partial B_{x_{n,\alpha}}\left(0\right)$ and
introduce the function $g_{n}(\theta)=\sum_{j=1}^{\alpha(n)}x_{n,\alpha}^j\frac{\mathrm{e}^{\mathrm{i}j\theta}-1}{j}$.
We then arrive at the expression
\begin{align*}
I_{n}= & \frac{\alpha\left(n\right)^{2}}{Z_{n,\alpha}n^{2}}\frac{\exp\left(\sum_{j=1}^{\alpha\left(n\right)}\frac{x_{n,\alpha}^{j}}{j}\right)}{2\pi x_{n,\alpha}^{n}}\int_{-\pi}^{\pi}G_{n,t_{1},t}\left(x_{n,\alpha}\mathrm{e}^{\mathrm{i}\theta}\right)G_{n,t,t_{2}}\left(x_{n,\alpha}\mathrm{e}^{\mathrm{i}\theta}\right)\exp\left(g_{n}\left(\theta\right)\right)\mathrm{d}\theta.
\end{align*}
We also split the integral into two parts. The main contribution
is again due to the interval $\left[-\theta_n,\theta_n\right]$
. By Lemma \ref{lem:partLambda0}, literally retracing the
steps in the proof of Proposition \ref{prop:GenSaddle} shows that
\[
\frac{\alpha\left(n\right)^{2}}{Z_{n,\alpha}n^{2}}\frac{\exp\left(\sum_{j=1}^{\alpha\left(n\right)}\frac{x_{n,\alpha}^{j}}{j}\right)}{2\pi x_{n,\alpha}^{n}}\int_{\pi\geq \left|\theta\right|>\theta_n}G_{n,t_{1},t}\left(x_{n,\alpha}\mathrm{e}^{\mathrm{i}\theta}\right)G_{n,t,t_{2}}\left(x_{n,\alpha}\mathrm{e}^{\mathrm{i}\theta}\right)\exp\left(g_{n}\left(\theta\right)\right)\mathrm{d}\theta
\]
vanishes faster than any power of $1/n$. It poses no problem due to $t_2-t_1 \geq \log(n/\alpha(n))/\alpha(n)$.
For $|\theta|\leq \theta_n$, apply 
$
|\mathrm{e}^{\mathrm{i}j\theta}-1| \leq c_1 j\theta
$
for some $c_1 >0$ for all $j$ and $|\mathrm{e}^{\mathrm{i}j\theta}|=1$. Then there is $c_2 >0$ such that
\begin{align*}
&\left| G_{n,t_{1},t}\left(x_{n,\alpha}\mathrm{e}^{\mathrm{i}\theta}\right)G_{n,t,t_{2}}\left(x_{n,\alpha}\right)\right|\\
\leq & c_2 \left(\left(\theta\sum_{j=b_{t_{1}}\left(n\right)+1}^{b_{t}\left(n\right)} x_{n,\alpha}^{j}\right)^{2}+\sum_{j=b_{t_{1}}\left(n\right)+1}^{b_{t}\left(n\right)}\frac{x_{n,\alpha}^{j}}{j}\right)
\left(\left(\theta\sum_{j=b_{t}\left(n\right)+1}^{b_{t_2}\left(n\right)} x_{n,\alpha}^{j}\right)^{2}+\sum_{j=b_{t}\left(n\right)+1}^{b_{t_2}\left(n\right)}\frac{x_{n,\alpha}^{j}}{j}\right)
\end{align*}
for all $n$. Due to equation \eqref{eq:gN}, we have
$
|\exp\left(g_n(\theta)\right)|\leq c_3 \exp\left(-\frac{\lambda_{2,n}}{2}\theta^2 \right)
$
for some $c_3>0$ and all $|\theta|\leq \theta_n$ if $n$ is large enough. By substituting $v=\sqrt{\lambda_{2,n}}\theta$, we therefore obtain
$$
\left|\int_{-\theta_n}^{\theta_n}\theta^{k}\exp\left(g_{n}\left(\theta\right)\right)\mathrm{d}\theta\right|\leq c_4 \lambda_{2,n}^{-\frac{k+1}{2}}
$$
for some $c_4>0$ and $0\leq k \leq 4$ because of the moments of
the normal distribution. By linearity of the integral as well as the
definition of $Z_{n,\alpha}$ and Lemmata \ref{lem:Lambda2} and \ref{lem:StraffHilfe}, we conclude
\begin{align*}
 I_n \leq  c^\prime \left[\left(t-t_{1}\right)^{2}+t-t_{1}\right]\left[\left(t_{2}-t\right)^{2}+\left(t_{2}-t\right)\right] \leq c \left(t_{2}-t_{1}\right)^{2}
\end{align*}
for some $c^\prime, c>0$ and $n$ large enough. The last step holds due to $\delta \leq t_{1}\leq t\leq t_{2}\leq1$. 
\end{proof}

\subsection{Properties of $h_n$\label{sub:Beweise}}
This section provides the proofs for five properties of $h_n$ and its derivatives stated in Section \ref{sec:limitshape}. We are going to need the asymptotics presented in
\begin{lem}
\label{lem:partLambda0}Let $0<t\leq1$. Then,
\begin{equation}
\sum_{j=1}^{b_{t}\left(n\right)}x_{n,\alpha}^{j}\sim tn\text{ and }\sum_{j=1}^{b_{t}\left(n\right)}\frac{x_{n,\alpha}^{j}}{j}\sim t\frac{n}{\alpha\left(n\right)}
\end{equation}
hold.
\end{lem}
\begin{proof}
Since $x_{n,\alpha}>1$, we have
\begin{align}
\int_{0}^{b_{t}\left(n\right)}x_{n,\alpha}^{v}\mathrm{d}v\leq\sum_{j=1}^{b_{t}\left(n\right)}x_{n,\alpha}^{j}\leq\int_{1}^{b_{t}\left(n\right)+1}x_{n,\alpha}^{v}\mathrm{d}v \sim \int_{0}^{b_{t}\left(n\right)}x_{n,\alpha}^{v}\mathrm{d}v\label{eq:partHilfe}
\end{align}
by Lemma \ref{lem:StaSad}.
It therefore remains to be shown that
$
\int_{0}^{b_{t}\left(n\right)}x_{n,\alpha}^{v}\mathrm{d}v=  \frac{\left(x_{n,\alpha}\right)^{b_{t}\left(n\right)}-1}{\log\left(x_{n,\alpha}\right)}
\sim  tn.
$
Since
$
0<\frac{b_{t}\left(n\right)}{\alpha\left(n\right)}\leq 1
$
for $n$ large enough, the first claim follows from equation \eqref{eq:partHilfe} and
\begin{align*}
\left(x_{n,\alpha}\right)^{b_{t}\left(n\right)}=  \left[\left(x_{n,\alpha}\right)^{\alpha\left(n\right)}\right]^{\frac{b_{t}\left(n\right)}{\alpha\left(n\right)}}
\sim  \exp\left[\frac{b_{t}\left(n\right)}{\alpha\left(n\right)}\log\left(\frac{n}{\alpha\left(n\right)}\log\left(\frac{n}{\alpha\left(n\right)}\right)\right)\right]
\sim  t\frac{n}{\alpha\left(n\right)}\log\left(\frac{n}{\alpha\left(n\right)}\right),\label{eq:kNpower}
\end{align*}
which holds due to Lemma \ref{lem:StaSad}.
It was proved in
Proposition 4.8 in \cite{BS17} that
$\sum_{j=1}^{\alpha(n)}\frac{x_{n,\alpha}^{j}}{j}\sim \frac{n}{\alpha(n)}$.
Consider
$
\sum_{j=1}^{b_t(n)}\frac{x_{n,\alpha}^j}{j}=\sum_{j=1}^{\alpha(n)}\frac{x_{n,\alpha}^j}{j}-\sum_{j=b_t(n)+1}^{\alpha(n)}\frac{x_{n,\alpha}^j}{j}.
$
Due to $b_t(n)\sim \alpha(n)$ and the first claim,
\[
\frac{1}{\alpha(n)} \sum_{j=b_t(n)+1}^{\alpha(n)}x_{n,\alpha}^j           \leq    \sum_{j=b_t(n)+1}^{\alpha(n)}\frac{x_{n,\alpha}^j}{j} \leq \frac{1}{b_t(n)+1} \sum_{j=b_t(n)+1}^{\alpha(n)}x_{n,\alpha}^j   
\]
yields
%
the second claim.
\end{proof}
Let  $\gamma\left(n\right)\geq\log (n)$, $\boldsymbol{t}=\left(t_{1},...,t_{m}\right)^{T}$
for $m\in\mathbb{N}$ and $h_n(\bss)$ as in \eqref{eq:h(s)} throughout this section.
Set further $t_0=0$ and $t_{m+1}=1$. 
%
Property (i), which states that $h_n$ is infinitely often differentiable in $\bss$, follows from the differentiability of the saddle point $x_{{n,\alpha,\gamma,\boldsymbol{t}}}$ which can be shown by applying the implicit function theorem to the function
%
$
F\left(\boldsymbol{s},x\right)=\sum_{i=0}^{m}\sum_{j=b_{t_{i}}\left(n\right)+1}^{b_{t_{i+1}}\left(n\right)}\left[\exp\left(\sum_{l=i+1}^{m}\frac{s_l}{\gamma (n)}\right)x\right]^{j}-n,
$
see \eqref{eq:NewSaddleLimitshape}.
%
So we can compute the derivatives of $h_n$.\\ 
Fix $i_3 \leq i_2 \leq i_1$ and let $x_{n}\left(\boldsymbol{s}\right) :=x_{n,\alpha,\gamma,\boldsymbol{t}}\left(\boldsymbol{s}\right)$. For the sake of brevity, we introduce the notations
$$
\lambda_{p,n}^{(i_1)}:=\sum_{i=0}^{i_1-1}\mathrm{e}^{\sum_{l=i+1}^{m}\frac{s_l}{\gamma (n)}}\sum_{j=b_{t_{i}}\left(n\right)+1}^{b_{t_{i+1}}\left(n\right)}
j^{p-1}\left(x_{n}\left(\boldsymbol{s}\right)\right)^{j}
$$
so that $\lambda_{p,n}=\lambda_{p,n}^{(m+1)}$. We obtain
\begin{align}
\partial_{s_{i_{1}}}h_{n}\left(\boldsymbol{s}\right)= &\frac{1}{\gamma(n)}\lambda_{0,n}^{(i_1)},\label{eq:h1}\\
\partial_{s_{i_{2}}}\partial_{s_{i_{1}}}h_{n}\left(\boldsymbol{s}\right)= &\frac{1}{\gamma(n)} \frac{\partial_{s_{i_{2}}}x_{n}\left(\boldsymbol{s}\right)}
{x_{n}\left(\boldsymbol{s}\right)}\lambda_{1,n}^{(i_1)} +\frac{1}{(\gamma(n))^2} \lambda_{0,n}^{(i_2)}\label{eq:h2}
\end{align}
and
\begin{align}
  \partial_{s_{i_{3}}}\partial_{s_{i_{2}}}\partial_{s_{i_{1}}}h_{n}\left(\boldsymbol{s}\right)= &\frac{1}{(\gamma(n))^2} \frac{\partial_{s_{i_{2}}}x_{n}\left(\boldsymbol{s}\right)}{x_{n}\left(\boldsymbol{s}\right)} \lambda_{1,n}^{(i_3)}\label{eq:h3}\\
 & +\frac{1}{\gamma(n)}\left(\frac{\partial_{s_{i_{3}}}\partial_{s_{i_{2}}}x_{n}\left(\boldsymbol{s}\right)}{x_{n}\left(\boldsymbol{s}\right)}-\frac{\partial_{s_{i_{2}}}x_{n}\left(\boldsymbol{s}\right)}{x_{n}\left(\boldsymbol{s}\right)}\frac{\partial_{s_{i_{3}}}x_{n}\left(\boldsymbol{s}\right)}{x_{n}\left(\boldsymbol{s}\right)}\right)  \lambda_{1,n}^{(i_1)}
\nonumber\\
 & +\frac{1}{\gamma(n)}\frac{\partial_{s_{i_{2}}}x_{n}\left(\boldsymbol{s}\right)}{x_{n}\left(\boldsymbol{s}\right)}\frac{\partial_{s_{i_{3}}}x_{n}\left(\boldsymbol{s}\right)}{x_{n}\left(\boldsymbol{s}\right)}  \lambda_{2,n}^{(i_1)}
 +\frac{1}{(\gamma(n))^2}\frac{\partial_{s_{i_{3}}}x_{n}\left(\boldsymbol{s}\right)}{x_{n}\left(\boldsymbol{s}\right)} \lambda_{1,n}^{(i_2)}
  + \frac{1}{(\gamma(n))^3}\lambda_{0,n}^{(i_3)}.\nonumber
\end{align}
In order to prove properties (ii) to (v), we need to understand the derivatives of the saddle point.
\begin{lem}
\label{lem:Xdev}
Fix $i_{2}\leq i_{1}$. Then,
\[
\frac{\partial_{s_{i_{1}}}x_{n}\left(\boldsymbol{s}\right)}{x_{n}\left(\boldsymbol{s}\right)}=-\frac{1}{\gamma(n)}\frac{\lambda_{1,n}^{(i_1)}}
{\lambda_{2,n}}.
\]
Moreover,
\[
\frac{\partial_{s_{i_{1}}}x_{n}\left(\boldsymbol{s}\right)}{x_{n}\left(\boldsymbol{s}\right)}=\caO\left(\frac{1}{\gamma(n)\alpha(n)}\right)
\text{ and }
\frac{\partial_{s_{i_{2}}}\partial_{s_{i_{1}}}x_{n}\left(\boldsymbol{s}\right)}{x_{n}\left(\boldsymbol{s}\right)}-\frac{\partial_{s_{i_{2}}}x_{n}\left(\boldsymbol{s}\right)}{x_{n}\left(\boldsymbol{s}\right)}\frac{\partial_{s_{i_{1}}}x_{n}\left(\boldsymbol{s}\right)}{x_{n}\left(\boldsymbol{s}\right)}=\caO\left(\frac{1}{(\gamma(n))^2\alpha\left(n\right)}\right)
\]
hold locally uniformly in $\boldsymbol{s}$.
\end{lem}
\begin{proof}
Differentiating equation (\ref{eq:NewSaddleLimitshape}) with respect to $s_{i_{1}}$ yields
$
0=\frac{1}{\gamma(n)}\lambda_{1,n}^{(i_1)}
+\frac{\partial_{s_{i_{1}}}x_n\left(\boldsymbol{s}\right)}{x_n\left(\boldsymbol{s}\right)}\lambda_{2,n},
$
so
\[
\frac{\partial_{s_{i_1}}x_{n}\left(\boldsymbol{s}\right)}{x_{n}\left(\boldsymbol{s}\right)}=-\frac{1}{\gamma(n)}\frac{\lambda_{1,n}^{(i_1)}}
{\lambda_{2,n}}
=\caO\left(\frac{1}{\gamma(n)\alpha(n)}\right)
\]
by equation \eqref{eq:NewSaddleLimitshape} and Lemma \ref{lem:Lambda2}. W.l.o.g., let $i_2 \leq i_1$.
Differentiating once more, now with respect to $s_{i_{2}}$, we obtain
\begin{align*}
  \frac{\partial_{s_{i_{2}}}\partial_{s_{i_{1}}}x_{n}\left(\boldsymbol{s}\right)}{x_{n}\left(\boldsymbol{s}\right)}-\frac{\partial_{s_{i_{2}}}x_{n}\left(\boldsymbol{s}\right)}{x_{n}\left(\boldsymbol{s}\right)}\frac{\partial_{s_{i_{1}}}x_{n}\left(\boldsymbol{s}\right)}{x_{n}\left(\boldsymbol{s}\right)}
= & -\frac{1}{(\gamma(n))^2}\frac{\lambda_{1,n}^{(i_2)}}{\lambda_{2,n}}
  -\frac{1}{\gamma(n)}\frac{\partial_{s_{i_{2}}}x_{n}\left(\boldsymbol{s}\right)}{x_{n}\left(\boldsymbol{s}\right)}\frac{\lambda_{2,n}^{(i_1)}}{\lambda_{2,n}}\\
 & +\frac{1}{(\gamma(n))^2}\frac{\lambda_{1,n}^{(i_1)}}
{\left(\lambda_{2,n}\right)^{2}}
\lambda_{2,n}^{(i_2)}
 +\frac{1}{\gamma(n)}\frac{\partial_{s_{i_{2}}}x_{n}\left(\boldsymbol{s}\right)}{x_{n}\left(\boldsymbol{s}\right)}\frac{\lambda_{1,n}^{(i_1)}}
{\left(\lambda_{2,n}\right)^{2}}
\lambda_{3,n}.
\end{align*}
Applying Lemma \ref{lem:Lambda2}, equation (\ref{eq:LambdaP}), and
the first result to each term, we conclude the last claim.
\end{proof}
Property (ii) is now a direct consequence of equation \eqref{eq:h1} and Lemma \ref{lem:partLambda0}, (iii) and (iv) follow from equation \eqref{eq:h2} and Lemmata \ref{lem:Xdev} and \ref{lem:partLambda0}. Property (v) can easily be deduced from equation \eqref{eq:h3} and Lemmata \ref{lem:Xdev} and \ref{lem:partLambda0}.

\section*{Acknowledgements}
H.S. acknowledges support by Deutsche Telekom Stiftung.

\bibliographystyle{plain}


\end{document}